\documentclass[hyphens,a4paper,11pt]{article}

\usepackage[utf8]{inputenc}
\usepackage{pgfplots}
\pgfplotsset{compat=1.18}
\usepgfplotslibrary{groupplots}
\usepackage{adjustbox}
\usepackage{tikz}
\usepackage{caption}

\usepackage{inputenc}
\usepackage{graphics}
\usepackage{epsfig}
\usepackage{graphicx}
\usepackage{latexsym}
\usepackage{graphics}
\usepackage{epsfig}
\usepackage{amsmath,amssymb,amsfonts,theorem,color,bm}
\usepackage{multirow}
\usepackage{verbdef}
\usepackage{mathrsfs}

\usepackage{subcaption}
\usepackage{epstopdf}
\usepackage{float}
\newcommand{\bs}{\boldsymbol}

\newcommand{\beqn}{\begin{equation}}
\newcommand{\eeqn}{\end{equation}}

\newtheorem{remark}{Remark}[section]

\newtheorem{theorem}{Theorem}[section]

\newtheorem{proposition}[theorem]{Proposition}

\def\blackbox{\leavevmode\vrule height 5pt width 4pt depth 0pt\relax}
\newenvironment{proof}{\begin{trivlist}
\item[]\hspace{0cm}{\bf Proof:}
\hspace{0cm} }{\hfill $\blackbox$
\end{trivlist}}

\newcommand\restr[2]{{
\left.\kern-\nulldelimiterspace 
#1 
\right|_{#2} 
}}
\usepackage{fix-cm}

\usepackage{amsfonts}
\usepackage{mathtools}

 \usepackage{xcolor}
 \usepackage{listings}
 
\usepackage{amssymb}

\usepackage{todonotes}

\usepackage{multirow}

\newcommand{\mK}{\mathsf{K}}

\newcommand{\mO}{\mathsf{O}}
\newcommand{\mI}{\mathsf{I}}

\newcommand{\mV}{\mathsf{V}}

\newcommand{\mTheta}{\mathsf{\Theta}}
\newcommand{\mSigma}{\mathsf{\Sigma}}
\newcommand{\mLambda}{\mathsf{\Lambda}}

\usepackage[affil-it]{authblk}
\usepackage{url}

\usepackage{soul}

\usepackage{pgfplots} 				
\usepackage{tikz}
\usepackage{tikzscale}
\usepackage{environ}
\newlength\fwidth

\definecolor{codegreen}{rgb}{0,0.6,0}
\definecolor{codegray}{rgb}{0.5,0.5,0.5}
\definecolor{codepurple}{rgb}{0.58,0,0.82}
\definecolor{backcolour}{rgb}{0.95,0.95,0.92}

\lstdefinestyle{mystyle}{
    backgroundcolor=\color{backcolour},   
    commentstyle=\color{codegreen},
    keywordstyle=\color{magenta},
    numberstyle=\tiny\color{codegray},
    stringstyle=\color{codepurple},
    basicstyle=\ttfamily\footnotesize,
    breakatwhitespace=false,         
    breaklines=true,                 
    captionpos=b,                    
    keepspaces=true,                 
    numbers=left,                    
    numbersep=5pt,                  
    showspaces=false,                
    showstringspaces=false,
    showtabs=false,                  
    tabsize=2
}

\title{Feature Understanding and Sparsity Enhancement via 2-Layered kernel machines (2L-FUSE)}
\author{F. Camattari$^{1,2,3}$, S. Guastavino$^{1,2,3}$, F. Marchetti$^{2,3,4}$, E. Perracchione$^{3,5}$\\
$^1$ MIDA, Dipartimento di Matematica, Università di Genova, Italy.\\
$^2$ Osservatorio Astrofisico di Torino, {Istituto Nazionale di Astrofisica}, Torino, {Italy}\\
$^3$ Gruppo Nazionale per il Calcolo Scientifico INdAM, Roma, Italy\\
$^4$ Dipartimento di Matematica \lq\lq Tullio Levi-Civita", Università di Padova,  {Italy}\\
$^5$ Dipartimento di Scienze Matematiche \lq\lq Giuseppe Luigi Lagrange\rq\rq, {Politecnico di Torino}, {Italy}}
\date{\today}

\begin{document}

\maketitle

\begin{abstract}
We propose a novel sparsity enhancement strategy for regression tasks, based on learning a data-adaptive kernel metric, i.e., a shape matrix, through 2-Layered kernel machines. The resulting shape matrix, which defines a Mahalanobis-type deformation of the input space, is then factorized via an eigen-decomposition, allowing us to identify the most informative directions in the space of features. This data-driven approach provides a flexible, interpretable and accurate feature reduction scheme. Numerical experiments on synthetic and applications to real datasets of geomagnetic storms demonstrate that our approach achieves minimal yet highly informative feature sets without losing predictive performance.
\end{abstract}

\section{Introduction}

We focus on supervised regression, where the goal is to approximate scattered data points in a feature space. In this context, sparsity enhancement plays a crucial role, as selecting the most relevant features is a key tool for both computational efficiency and model reliability. 
The most popular feature reduction procedures include Principal Component Analysis (PCA) \cite{jolliffe2002pca}, Lasso regression \cite{lasso}, or its variations, such as, e.g., Group Lasso \cite{yuan2006group}, Adaptive Lasso \cite{zou2006adaptive} and (Adaptive) Poisson Re-weighted Lasso \cite{guastavino2019consistent}. 

However, these methods often struggle to select relevant features when complex nonlinear dependencies are present. 
Appealing alternatives arise from kernel-based methods, which are naturally suited to capture nonlinear relationships between variables. Within this framework, several strategies have been explored, including greedy selection methods \cite{Camattari}, kernel PCA \cite{scholkopf1997kernelpca}, and feature ranking based on Support Vector Machines (SVMs) \cite{Guyon}. These approaches leverage the flexibility of kernels and lead to more interpretable feature selection in non-linear settings. However, we remark that the accuracy of the kernel models is known to depend on the scaling of the data. Hence, particularly for regression tasks, many efforts in literature are devoted to find suitable, i.e., safe, values for the shape parameter; see, e.g. \cite{cavoretto2021search,ling2022stochastic,mojarrad2023new,rippa1999algorithm,scheuerer2011alternative}. Moreover, other works deal with the more challenging problem of learning a \emph{shape matrix} \cite{aiolli2014learning} or a \emph{shape function}, based on intuitions on the Lebesgue constant and its generalization \cite{brunidrna,brunisisc}, instead of a single shape parameter \cite{audonekernel}. In this adaptive context, the 2-Layered kernel machines offer an elegant and flexible framework for learning the \emph{optimal} kernel metric \cite{Doppel2024759,Wenzel}: the learned kernel is not only anisotropic but allows rotations of the kernel basis that aligns more closely with the underlying structure of the data. 

In this paper, we propose a promising direction for sparsity enhancement in kernel-based models which, as first step, involves learning a data-adaptive kernel metric, i.e. a shape matrix, rather than relying on a fixed shape parameter. After theoretically studying the convergence properties of the so-constructed regressor, we perform an eigen-decomposition to the shape matrix which defines the optimal metric. This decomposition enables a Feature Understanding and Sparsity Enhancement via 2-Layered kernel machines (2L-FUSE). For the above reasons, the 2L-FUSE is more flexible than both SVM-based feature selection and kernel PCA that are typically trained with isotropic kernels. Moreover, unlike Kernel PCA, which performs dimensionality reduction by applying spectral decomposition to the kernel matrix, our approach learns a task-specific kernel metric and then performs reduction directly on the learned shape matrix. 

Furthermore, we are able to prove that reducing the features via the 2L-FUSE strategy is equivalent to perform a standard kernel regression with the original features mapped accordingly to the optimal recovered metric, leading to a simple and efficient implementation of the 2L-FUSE.  
These theoretical findings are confirmed by numerical tests carried out with both synthetic and real-world datasets of geomagnetic storm occurrence, showing that we are able to find a \emph{minimal} set of features without any accuracy loss. 

The paper is organized as follows. Section \ref{Preliminari} briefly recalls the basics of kernel interpolation. Section \ref{2lfuse} is the core of this study; it presents the theoretical results about the 2L-FUSE algorithm. Experiments and conclusions are offered respectively in Sections \ref{numerics} and \ref{conclusions}.

\section{Preliminaries}
\label{Preliminari}

For simplicity, we focus on kernel-based interpolation tasks, showing later that the same outcomes can be trivially extended to regression models.  
The scattered data interpolation problem can be formulated as follows. Given $X = \{ \boldsymbol{x}_1,\ldots,\boldsymbol{x}_n \} \subseteq \Omega$ a non-empty set of pairwise distinct interpolation points with $\Omega \subset \mathbb{R}^d$, $d \geq 1, d \in \mathbb{N}$ and $F =  \{f(\boldsymbol{x}_1), \ldots, f(\boldsymbol{x}_n)\}=\{f_1,\ldots,f_n\} \subseteq \mathbb{R}$ the set of  associated function values, find an approximation $s:\Omega \longrightarrow \mathbb{R}$ of the (unknown) function $f$ so that: 
\begin{equation}\label{interp_cond}
    s(\boldsymbol{x}_i) = f(\boldsymbol{x}_i), \quad i=1,\ldots,n. 
\end{equation}

As the scattered data interpolation problem might be in principle defined for \emph{large} dimensions $d$, it is worth considering meshless techniques \cite{fasshauer2007meshfree,fasshauer2015kernel,wendland2005scattered}. For instance, radial kernel-based schemes have the advantage of being easy to implement in any dimensions; indeed to any radial kernel we can associate a univariate function that only depends on the distance among points. More precisely, to each strictly positive definite kernel $\kappa: \Omega \times  \Omega \longrightarrow \mathbb{R}$ we can associate a unique Reproducing Kernel Hilbert Space $N_{\kappa} (\Omega)$ (RKHS), also known as native space, equipped with an inner product $\langle \cdot, \cdot \rangle$ that is made of  functions $f:\Omega \rightarrow \mathbb{R}$ for which $\kappa$ acts as a reproducing kernel, i.e.: 
\begin{itemize}
    \item $\kappa(\cdot, \boldsymbol{x}) \in N_{\kappa}(\Omega)$,  $ \boldsymbol{x} \in \Omega$,
    \item $f(\boldsymbol{x})=\langle f, \kappa(\cdot,\boldsymbol{x}) \rangle$, $ \boldsymbol{x} \in \Omega$, $ f \in N_{\kappa}(\Omega)$.
\end{itemize}

If we restrict to radial strictly positive definite kernels, then we can consider univariate functions known as Radial Basis Functions (RBFs), meaning that, we are able to uniquely associate to the kernel a function $ \phi: \mathbb{R}_{+} \longrightarrow\mathbb{R}$, where $\mathbb{R}_{+}= [0,+\infty)$, and (possibly) a shape parameter $\varepsilon>0$ such that, for all $\boldsymbol{x},\boldsymbol{z}\in \Omega$, 
\begin{equation*} %\label{kerrbf}
    \kappa(\boldsymbol{x},\boldsymbol{z})= \kappa_{\varepsilon}(\boldsymbol{x},\boldsymbol{z})=  \phi( \sqrt{(\boldsymbol{x}-\boldsymbol{z})^{\intercal} \varepsilon^2 \cdot \mI_d (\boldsymbol{x}-\boldsymbol{z})}) = \phi( \varepsilon ||\boldsymbol{x}-\boldsymbol{z}||_2) =\phi_{\varepsilon}( ||\boldsymbol{x}-\boldsymbol{z}||_2)=  \phi(r),
\end{equation*}
where $\mI_d$ is the $d \times d$ identity matrix and $r=||\boldsymbol{x}-\boldsymbol{z}||_2$. In what follows we may omit the dependence of the kernel on the shape parameter.  

The scattered data interpolation problem then consists in finding the coefficients $\boldsymbol{c}=(c_1,\ldots,c_n)^{\intercal}$ of the interpolant 
\begin{equation*} 
    s(\cdot) \equiv \Pi_{V(X)}(f) = \sum_{i=1}^n c_i \kappa(\cdot,\boldsymbol{x}_i),
\end{equation*}
so that the conditions in equation \eqref{interp_cond} are satisfied and where $\Pi_{V(X)}(f)$ is the orthogonal projection onto the linear subspace $V(X) = \textrm{span} \{\kappa(\cdot, \boldsymbol{x}_i),\; \boldsymbol{x}_i \in  X \}$. Thus we need to solve a linear system of the form: 
\begin{equation*} 
    \mK  \boldsymbol{c} = \boldsymbol{f}, 
\end{equation*}
where $\boldsymbol{f}=(f_1,\ldots,f_n)^{\intercal}$ and 
\begin{equation*}
    \mK_{ij}=\kappa(\boldsymbol{x}_i,\boldsymbol{x}_j)= \phi \left( \sqrt{(\boldsymbol{x}_i-\boldsymbol{x}_j)^{\intercal} \varepsilon^2 \cdot \mI_d (\boldsymbol{x}_i-\boldsymbol{x}_j )}\right).
\end{equation*}
Under our assumptions, the kernel collocation matrix, which depends on the RBF and on the Euclidean distance  among the points, is non-singular and this in turn ensures that the interpolation problem has a unique solution, and we may write the nodal form of the interpolant as 
\begin{equation*}
     s\left( \boldsymbol{x}\right)=  \boldsymbol{\kappa}^{\intercal}({\boldsymbol{x}} ) \mK^{-1} \boldsymbol{f}, \quad \boldsymbol{x} \in \Omega,
\end{equation*}
where 
\begin{equation*}
    \boldsymbol{\kappa}^{\intercal}({\boldsymbol{x}} ) = \left(\kappa(\boldsymbol{x},\boldsymbol{x}_1), \ldots,\kappa (\boldsymbol{x},\boldsymbol{x}_n) \right).
\end{equation*}
    
Classical pointwise error bounds for kernel-based interpolants are of the form (see, e.g., \cite{wendland2005scattered}):
\begin{equation}\label{eq:bound_pow}
    |f(\boldsymbol{x})-s(\boldsymbol{x})| \leq P_{X} \| f - s \|_{N_{\kappa}(\Omega)}, \hskip 0.15cm \boldsymbol{x} \in \Omega, \hskip 0.15cm f \in N_{\kappa}(\Omega), 
\end{equation}
where $P_{X}$, known as \emph{power function}, is defined as 
\begin{align*}
    P_{X}(\boldsymbol{x}) = \| \kappa(\cdot, \boldsymbol{x})- \Pi_{V(X)}(\kappa(\cdot, \boldsymbol{x})) \|_{N_{\kappa}(\Omega)}.
\end{align*}

Other error indicators are based on the so-called fill-distance (which can be thought of as the mesh size for gridded data)  defined as
\begin{equation*}
 h_{X} =  \sup_{ \boldsymbol{x} \in \Omega} \left( \min_{ \boldsymbol{x}_k  \in {X}} \left\| \boldsymbol{x} - \boldsymbol{x}_k \right\|_2 \right).
\end{equation*}
Provided that $h_X$ is \textit{small enough} (see \cite[Theorem 3.2]{wendland2005scattered}), it relates to the pointwise error via the following inequality
\begin{equation}
	|f\left(\boldsymbol{x}\right)-s\left(\boldsymbol{x}\right) | \leq C h^{\gamma}_{X}  ||f||_{N_{\kappa}(\Omega)}, \quad \boldsymbol{x} \in \Omega,\hskip 0.15cm f \in N_{\kappa}(\Omega), 
\label{eq:bound_fill}
\end{equation}
where $\gamma>0$ depends on the kernel's smoothness.

In the next section, we will take advantage of the technique introduced in \cite{Wenzel}, namely the 2-Layered kernel machine, and we show how it naturally leads to sparsity enhancement. Such a method, called 2L-FUSE, will be studied in the following and later tested on various datasets. 

\section{2L-FUSE}
\label{2lfuse}

In order to introduce the 2L-FUSE idea, we remark that any isotropic radial kernel can be turned into an anisotropic one by using a weighted $2-$norm instead of an unweighted one. Indeed, it is enough to replace the scalar value of the shape parameter
with a diagonal matrix $\mSigma$. Doing in this way, different scalings along the $d$ dimensions are allowed. Nevertheless, if the matrix $\mSigma$ is diagonal, rotations of the kernel basis are not allowed. To accomplish this, one should find the \emph{optimal} shape matrix $\mTheta$, being $\mTheta$ a  $d \times d$ positive definite matrix such that $\mTheta=\mSigma^{\intercal}\mSigma$, with
\begin{equation*}
    \mSigma=
    \begin{pmatrix}
        \varepsilon_{11} & \dots & \varepsilon_{1d}\\
        \vdots & \ddots & \vdots \\
        \varepsilon_{d1} & \dots & \varepsilon_{dd}
    \end{pmatrix}.
\end{equation*}
The associated kernel, rescaled and rotated accordingly with $\mSigma$, is denoted by $\kappa_{\mSigma}$ and is so that
   \begin{equation} \label{eq:linear_two_layered_kernel}
 \kappa_{\mSigma}(\boldsymbol{x},\boldsymbol{z})= \phi( \sqrt{(\boldsymbol{x}-\boldsymbol{z})^{\intercal} \mSigma^{\intercal} \mSigma  (\boldsymbol{x}-\boldsymbol{z})}), \quad \boldsymbol{x}, \boldsymbol{z} \in \Omega.
    \end{equation}

In \cite{Wenzel}, by using as loss function the cross validation error \cite{marchetti2021extension}, the authors show that they are able to efficiently optimize the matrix $\mSigma$ and that $\kappa_{\mSigma}$ in \eqref{eq:linear_two_layered_kernel} is an instance of a two-layered kernel according to the deep kernel representer \cite[Theorem 1]{bohn2019representer}. 

Then, the 2L algorithm returns an optimized shape matrix which is used to define the interpolant $s_{\mSigma}$ (based on the the Mahalanobis distance) as
\begin{equation} \label{eq:optimal_kernel}
    s_{\mSigma}(\boldsymbol{x})  = \sum_{i=1}^n c^{\mSigma}_i \kappa_{\mSigma}(\boldsymbol{x},\boldsymbol{x}_i) = \sum_{i=1}^n c^{\mSigma}_i \kappa(\mSigma \boldsymbol{x}, \mSigma \boldsymbol{x}_i) = s( \mSigma \boldsymbol{x}), \quad \boldsymbol{x} \in \Omega, 
\end{equation}
where $\boldsymbol{c}^{\mSigma}=({c}_1^{\mSigma}, \ldots, {c}_n^{\mSigma})^{\intercal}$ is so that
\begin{equation} \label{eq:sys_sig}
    \mK_{\mSigma}  \boldsymbol{c}^{\mSigma} = \boldsymbol{f}, 
\end{equation}
with non-singular kernel matrix $(\mK_{\mSigma})_{ij} = \kappa(\mSigma \boldsymbol{x}_i, \mSigma \boldsymbol{x}_j) $, $i,j=1,\ldots, n$.  In the following, before introducing our idea of feature reduction schemes based on $s_{\mSigma}$, we first note that the legitimacy of the interpolation scheme is preserved, indeed we are able to re-interpret the results on the power function and fill-distance; refer to equations \eqref{eq:bound_pow} and \eqref{eq:bound_fill}.

\subsection{2L convergence properties}

In order to show that we are able to directly recover results in line with \eqref{eq:bound_pow} and \eqref{eq:bound_fill}, from now on we restrict ourselves to the case where $\mSigma$ is invertible. Then, using the kernel $\kappa_{\mSigma}$ on the original dataset $X$ is equivalent to employ a standard kernel on the set of nodes mapped by the linear transform $\mSigma$. Indeed, thanks to \eqref{eq:linear_two_layered_kernel} we have that $
\kappa({\mSigma}\boldsymbol{x},{\mSigma}\boldsymbol{z}) = \kappa_{\mSigma}(\boldsymbol{x},\boldsymbol{z})$. Moreover, letting $\Psi$ defined as the set $\Omega$ to which we apply the linear transform $\mSigma$, as the spaces ${\cal N}_{\kappa_{\mSigma}} (\Omega)$ and  ${\cal N}_{\kappa} (\Psi)$ are isometrically isomorphic \cite{Bozzini2015199,vskmpi}. 
Then, we can write a convergence bound for the mapped kernel, in the spirit of \eqref{eq:bound_fill}. We remark that the following result is an extension of \cite[Theorem 3.3]{Wenzel}, where the authors focused on the case of Matérn kernels.
\begin{proposition}
    Let $\kappa_{\mSigma} \in C^{2\gamma}(\Omega\times \Omega
	)$ be a strictly positive definite kernel. Then there exist positive constants $h_0$ and $C$, such that (see the link with \eqref{eq:bound_fill})
	\begin{equation*}
	| f\left(\boldsymbol{x}\right)- s_{\mSigma} \left(\boldsymbol{x}\right) |  \leq C h^{\gamma}_{ \mSigma X} \sqrt{C_{\kappa_{\mSigma}}(\boldsymbol{x})} \sqrt{\boldsymbol{f}^{\intercal}{\mK}_{\mSigma}^{-1} \boldsymbol{f}},
	\end{equation*}
	   provided that $0<h_{\mSigma X} \leq h_0$ and $f \in {\cal N}_{\kappa_{\mSigma}} (\Omega)$, 	where 
	\begin{equation*}
	C_{\kappa_{\mSigma}}(\boldsymbol{x}) = C_{\kappa}(\mSigma \boldsymbol{x}) = \max_{ |\boldsymbol{\beta} |=2\gamma} \left( \max_{\mSigma{\boldsymbol{v}},\mSigma {\boldsymbol{w}}  \in \Psi \cap B\left(\mSigma{\boldsymbol{x}},C_2h_{ \mSigma {X}}\right)} \left| D^{\boldsymbol{\beta} }_2  \kappa\left(\mSigma{\boldsymbol{v}},\mSigma{\boldsymbol{w}}\right) \right| \right),
	\end{equation*}
	where $C_2$ is a suitable constant defined in \cite[Theorem 14.4 p. 123]{fasshauer2007meshfree},  and
	where 
	$B\left(\mSigma{\boldsymbol{x}},C_2h_{ \mSigma{X}}\right)$ denotes the ball of radius $C_2h_{ \mSigma {X}_n}$ centred at $\mSigma {\boldsymbol{x}}$.
\end{proposition}
\begin{proof}
    Letting $\boldsymbol{\kappa}_{\mSigma}^{\intercal}({\boldsymbol{x}} ) = \left(\kappa_{\mSigma}(\boldsymbol{x},\boldsymbol{x}_1), \ldots,\kappa_{\mSigma} (\boldsymbol{x},\boldsymbol{x}_n) \right)$, the coefficients $\boldsymbol{{c}}_{\mSigma}={\mK}_{\mSigma}^{-1} \boldsymbol{f}$ are so that
 \begin{equation*}
 s_{\mSigma}(\bs{x}) = \boldsymbol{\kappa}_{\mSigma}^{\intercal}(\boldsymbol{x}) {\mK}_{\mSigma}^{-1}  \boldsymbol{f},
 \end{equation*}
 or, equivalently, because of the symmetry of the kernel matrix:
 \begin{equation}
 s_{\mSigma}(\bs{x}) = \boldsymbol{f}^{\intercal} {\mK}_{{\mSigma}}^{-1} \boldsymbol{k}_{\mSigma}(\boldsymbol{x}).
 \label{eq6}
 \end{equation}
 Moreover, thanks to the reproducing property, we have that 
 \begin{align*}
 \boldsymbol{f}^{\intercal} & = \left( \langle f,\kappa_{\mSigma}(\cdot,\boldsymbol{x}_1)\rangle_{{\cal N}_{\kappa_{\mSigma}} (\Omega)}, \ldots, \langle f,\kappa_{\mSigma}(\cdot,\boldsymbol{x}_N)\rangle_{{\cal N}_{\kappa_{\mSigma}} (\Omega)} \right),\\
 & = \langle f,(\kappa_{\mSigma}(\cdot,\boldsymbol{x}_1), \ldots, \kappa_{\mSigma}(\cdot,\boldsymbol{x}_N))\rangle_{{\cal N}_{\kappa_{\mSigma}} (\Omega)},\\
 & = \langle f, \boldsymbol{\kappa}_{\mSigma}^{\intercal}(\cdot)\rangle_{{\cal N}_{\kappa_{\mSigma}} (\Omega)}.
 \end{align*}
 By plugging this into \eqref{eq6}, we obtain
 \begin{equation*}
 s_{\mSigma}(\bs{x}) = \langle f, \boldsymbol{\kappa}_{\mSigma}^{\intercal}(\cdot)\rangle_{{\cal N}_{\kappa_{\mSigma}} (\Omega)} \mK_{\mSigma}^{-1} \boldsymbol{\kappa}_{\mSigma}(\boldsymbol{x}) = \langle f, \boldsymbol{\kappa}_{\mSigma}^{\intercal}(\cdot)\mK_{{\mSigma}}^{-1} \boldsymbol{\kappa}_{\mSigma}(\boldsymbol{x})\rangle_{{\cal N}_{\kappa_{\mSigma}} (\Omega)}.
 \label{eq7}
 \end{equation*}
Thus (see the link with \eqref{eq:bound_pow}), 
\begin{align*}
|s_{\mSigma}(\bs{x})-f(\bs{x})| & \leq ||f||_{{\cal N}_{\kappa_{\mSigma}} (\Omega)} || \kappa_{{\mSigma}}(\cdot,\boldsymbol{x})-\boldsymbol{\kappa}_{{\mSigma}}^{\intercal}(\cdot) \mK_{{\mSigma}}^{-1} \boldsymbol{\kappa}_{{\mSigma}}(\boldsymbol{x}) ||_{{\cal N}_{\kappa_{\mSigma}} (\Omega)} \nonumber \\
& = ||f||_{{\cal N}_{\kappa_{\mSigma}} (\Omega)} || \kappa(\cdot,\mSigma{\boldsymbol{x}})-\boldsymbol{\kappa}^{\intercal}(\cdot) \mK_{\mSigma}^{-1} \boldsymbol{\kappa}(\mSigma{\boldsymbol{x}}) ||_{{\cal N}_{\kappa} (\Psi)} \nonumber\\
& = ||f||_{{\cal N}_{\kappa_{\mSigma}} (\Omega)} {P}_{\mSigma X}(\mSigma {\boldsymbol{x}}). 
\end{align*}
Then, as shown in e.g. \cite[Theorem 14.5]{fasshauer2007meshfree}, we have that
\begin{align*}
{P}_{\mSigma X}(\mSigma {\boldsymbol{x}})  \leq  C h^{\gamma}_{ \mSigma{X}} \sqrt{C_{\kappa}(\mSigma{\boldsymbol{x}})},
\end{align*}
and the thesis follows by observing that $  ||f||_{{\cal N}_{\kappa_{\mSigma}} (\Omega)}  = \sqrt{\boldsymbol{f}^{\intercal}{\mK}_{\mSigma}^{-1} \boldsymbol{f}}$.
\end{proof} 
As a note, we point out that we can obtain faster convergence rates if $h_{\mSigma{X}} <h_{X}$, and a faster convergence order if $\mSigma$ is rank deficient (see \cite[Theorem 3.4]{Wenzel}). Nevertheless, the novelty in this work consists in exploiting properties of the optimized kernel matrix $\mK_{\mSigma}$ in order to infer the importance of the features in the dataset. %This turns out to be meaningful as kernel methods might be affected by the well-known curse of dimensionality and hence sparsity enhancement might be a solution. Moreover, as there is growing interest in machine learning and artificial intelligence, studying smart and efficient procedures for ranking the features and/or drastically reducing their number is a relevant issue. 

\subsection{2L for feature ranking}

In order to introduce our intuition behind the feature reduction scheme, we first recall that the matrix $\mTheta$ is positive definite, and hence we might write 
$$
\mTheta = \mV \mLambda \mV^{\intercal}, 
$$
where the columns of $\mV$ are the eigenvectors $\boldsymbol{v}_1^{\intercal}, \ldots, \boldsymbol{v}_d^{\intercal}$ of $\mTheta$ and the associated real positive eigenvalues $\lambda_{1}, \ldots, \lambda_d $ are stored (without loss of generality in a decreasing order) in the matrix $\mLambda= \mathrm{diag}(\lambda_{1}, \ldots, \lambda_{d})$. The key idea for sparsity enhancement lies in the fact that computing $s_{\Sigma}$ is equivalent to perform a standard kernel interpolation with the original features mapped accordingly to the optimal recovered metric. Indeed, we have the following result. 

\begin{proposition}
    Computing the interpolant with the kernel $\kappa_{\mSigma}$ as in \eqref{eq:optimal_kernel} is equivalent to computing it with a standard kernel on the the mapped points 
    \begin{align*}
        \bar{\boldsymbol{x}}_i = ({\bar{x}}_{i}^1, \ldots, {\bar{x}}_{i}^d)^{\intercal} & =  \left(\sum_{j=1}^d \sqrt{\lambda_1}v_1^j{{x}}_{i}^j, \ldots, \sum_{j=1}^d \sqrt{\lambda_d}v_d^j{{x}}_{i}^j\right)^{\intercal} \\  = &  \left( \sqrt{\lambda}_1{\boldsymbol{v}}_1^{\intercal} {\boldsymbol{x}}_i, \ldots, \sqrt{\lambda}_d{\boldsymbol{v}}_d^{\intercal} {\boldsymbol{x}}_i\right )^{\intercal}, \quad i= 1,\ldots,n.
    \end{align*}
\end{proposition}
\begin{proof}
The proof is straightforward, indeed for $i=1,\ldots,n$, we have that 
\begin{align*}
   & s_{\mSigma}(\boldsymbol{x}) = \sum_{i=1}^n  c_i^{\mSigma} \kappa_{\mSigma}(\boldsymbol{x},\boldsymbol{x}_i) = \sum_{i=1}^n  c_i^{\mSigma} \phi(\sqrt{(\boldsymbol{x}-\boldsymbol{x}_i)^{\intercal} \mTheta (\boldsymbol{x}-\boldsymbol{x}_i)}) = \\
   & \sum_{i=1}^n  c_i^{\mSigma} \phi(\sqrt{(\boldsymbol{x}-\boldsymbol{x}_i)^{\intercal} \mV \mLambda \mV^{\intercal} (\boldsymbol{x}-\boldsymbol{x}_i)}) = \\
    & \sum_{i=1}^n  c_i^{\mSigma} \phi\left(\sqrt{\left( \sum_{j=1}^d \sqrt{\lambda_1}v_1^j({{x}}^j-{{x}}_{i}^j), \ldots, \sum_{j=1}^d \sqrt{\lambda_d}v_d^j({{x}}^j-{{x}}_{i}^j)\right)\begin{pmatrix}
         \sum_{j=1}^d \sqrt{\lambda_1}v_1^j({{x}}^j-{{x}}_{i}^j) \\ 
         \ldots \\
         \sum_{j=1}^d \sqrt{\lambda_d}v_d^j({{x}}^j-{{x}}_{i}^j))
    \end{pmatrix} } \right) = \\
    &\sum_{i=1}^n  c_i^{\mSigma} \phi\left(\sqrt{\left( \sqrt{\lambda}_1{\boldsymbol{v}}_1^{\intercal} {(\boldsymbol{x}-\boldsymbol{x}_i)}, \ldots, \sqrt{\lambda}_d{\boldsymbol{v}}_d^{\intercal} {(\boldsymbol{x}-\boldsymbol{x}_i)}\right ) \begin{pmatrix}
    \sqrt{\lambda}_1{\boldsymbol{v}}_1^{\intercal} {(\boldsymbol{x}-\boldsymbol{x}_i)} \\ 
    \ldots \\
    \sqrt{\lambda}_d{\boldsymbol{v}}_d^{\intercal} {(\boldsymbol{x}-\boldsymbol{x}_i)}
    \end{pmatrix}} \right) =\\
     & \sum_{i=1}^n  c_i^{\mSigma} \phi(\sqrt{(\bar{\boldsymbol{x}}-\bar{\boldsymbol{x}}_i)^{\intercal} (\bar{\boldsymbol{x}}-\bar{\boldsymbol{x}}_i)}) = \sum_{i=1}^n  c_i^{\mSigma} \kappa
    (\bar{\boldsymbol{x}},\bar{\boldsymbol{x}_i}) = s(\bar{\boldsymbol{x}}).
\end{align*}
\end{proof}

In view of the above proposition, the idea of our dimensionality reduction easily follows. Indeed, if there exists an index $p$, $1\leq p < d$ so that $\lambda_{1} \geq \ldots \lambda_{p} \gg \lambda_{p+1} \geq \ldots \geq \lambda_{d}$, as usually done in each compression technique, we can set $\lambda_{p+1} = \ldots = \lambda_{d}=0$. In this way instead of taking all the features, we restrict to $p$ of them defined as linear combination of eigenvectors and original features properly scaled accordingly to the associated eigenvalues, i.e., we limit the analysis to $\mathbb{R}^p$ defined as $\tilde{\boldsymbol{x}}_i = ({\bar{x}}_{i}^1, \ldots, {\bar{x}}_{i}^p)^{\intercal} = \left( \sum_{j=1}^d \sqrt{\lambda_1}v_1^j{{x}}_{i}^j, \ldots, \sum_{j=1}^d \sqrt{\lambda_p}v_p^j{{x}}_{i}^j\right)^{\intercal}$, $i=1,\ldots,n$. In this case, we obtain an interpolant defined on a reduced number of features, formally defined as
$$
s_{\tilde{\mSigma}} (\boldsymbol{x}) = \sum_{i=1}^n c_i^{\tilde{\mSigma}}  \phi \left( \sqrt{(\boldsymbol{x}-\boldsymbol{x}_i)^{\intercal} \tilde{\mTheta} (\boldsymbol{x}-\boldsymbol{x}_i)}\right) = \sum_{i=1}^n c_i^{\tilde{\mSigma}}  \phi \left( \sqrt{(\tilde{\boldsymbol{x}}-\tilde{\boldsymbol{x}}_i)^{\intercal} (\tilde{\boldsymbol{x}}-\tilde{\boldsymbol{x}}_i)} \right) = s(\tilde{\boldsymbol{x}}), 
$$
where
$$
\tilde{\mTheta} =  \mV \begin{pmatrix}
\tilde{\mLambda} &  \\
 & \mO
\end{pmatrix}
{\mV}^{\intercal}, 
$$
being $\tilde{\mLambda} = \textrm{diag}(\lambda_1,\ldots,\lambda_p)$, and $\mO $ the ${(d-p) \times (d-p)}$ null matrix. Note that for the particular case of $p=1$, the mapped and reduced dataset would be defined by the only feature $\sum_{j=1}^d \sqrt{\lambda_1}v_1^j{{x}}_{i}^j$.

%Hence, as a bonus of what investigated in \cite{Wenzel}, we able to define a feature-reduction method for the optimized kernel, which is able to incorporate rotations in the space. 

\begin{remark}[The diagonal case]
\label{diag}
    Although learning the Mahalanobis distance is more challenging than optimizing $\mTheta$ in the special case of anisotropic kernels, that is, when $\mTheta$ is a diagonal matrix, in the context of feature reduction, it might provide useful information on the relevance of such features. Indeed, we first recall that for the special choices of $\mSigma = \varepsilon \cdot \mI$ we get $\mTheta= \varepsilon^2 \mI $, which provides us the usual kernels depending on a shape parameter. Moreover, when $\mSigma = \mathrm{diag}(\varepsilon_{11}, \ldots, \varepsilon_{dd}):=\mathrm{diag}(\varepsilon_{1}, \ldots, \varepsilon_{d})$. We are then able to optimize the scaling in the $d$ directions by learning $\varepsilon_i$, $i=1,\ldots,d$. Finally, we could re-interpret such a result as a feature reduction scheme as, without loosing generalities we suppose that the optimized parameters are so that $\varepsilon_{1}^2 \geq \ldots \geq \varepsilon_{d}^2$. If there exists an index $p$, $1\leq p < d$ so that $\varepsilon_{1}^2 \geq \ldots \varepsilon_{p}^2 \gg \varepsilon_{p+1}^2 \geq \ldots \geq \varepsilon_{d}^2$, we can again set $\varepsilon_{p+1} = \ldots = \varepsilon_{d}=0$, which is equivalent to discarding the last $d-p$ features in the original dataset. 
\end{remark}

We conclude this section by pointing out that what we have studied for the interpolation setting can be trivially extended to both ridge regression and Kriging approximation.

\begin{remark}[Real/noisy data and connection to ridge regression] 
\label{remark_noise} In this section, we deliberately focused only on the interpolation setting, as  tools known as \emph{smooting splines, ridge regression} and \emph{Tikhonov regularization}, which are typically defined by solving \cite{Tikhonov,Wahba}
\begin{equation*}
    (\mK_{\mSigma} + \lambda \mI) \boldsymbol{c}^{\mSigma} = \boldsymbol{f}, 
\end{equation*}
instead of \eqref{eq:sys_sig}, where $\lambda \in \mathbb{R}_+$, are equivalent to interpolating  the \emph{smoothed} data  $\hat{\boldsymbol{f}}= (\mK_{\mSigma} + \lambda \mI)^{-1} \mK \boldsymbol{f}$.
\end{remark}

In the following section, we numerically show our claims, i.e., how the 2L-FUSE can effectively used as feature reduction scheme and we test its performances in regression tasks both with synthetic datasets and real-world samples. 

\section{Numerical experiments}
\label{numerics}

Our experiments aim to demonstrate the effectiveness of the 2L-FUSE strategy in compressing the original feature space while preserving, or even enhancing, regression performances. To this end, we compare the predictive accuracy achieved using the full set of original features against that obtained through 2L-FUSE-based dimensionality reduction scheme. In the latter case we address the problem for both diagonal matrices (see Remark \ref{diag}) and full matrices $\mSigma$ optimized accordingly with the  Mahalanobis induced metric. In the 2L-FUSE setting, we compute the eigenvalues of the learned shape matrix $\mTheta$ and retain only the linear combination of eigenvectors so that the associated eigenvalues are above a certain threshold. 
% In the 2L-FUSE setting, we compute the eigenvalues of the learned shape matrix $\mTheta$ and retain only the linear combination provided by the eigenvectors so that the associated eigenvalues are above a certain threshold.

In all experiments the samples \(X \times F=\{(\boldsymbol{x}_i,f_i)\}_{i=1}^n \subset \Omega \times \mathbb{R}\)\} are randomly split into 80$\%$ training and 20$\%$ test sets. Then, the 2L-FUSE experiments are carried out accordingly with the following pipeline:
\begin{enumerate}
    \item For a given kernel \(\kappa\), we optimize the matrix \(\mSigma\) on the training data, for both the diagonal and non-diagonal cases. This is achieved thanks to the Python package freely available at \url{https://gitlab.mathematik.uni-stuttgart.de/pub/ians-anm/paper-2023-data-driven-kernel-designs};
    \item Reduce the features thanks to the eigenvalues analysis of $\mTheta$, i.e., construct the sets of mapped points $\tilde{X}=\{\tilde{x}_i\}_{i=1}^n$ in dimension $p$ (with $1\leq p \leq d$);% where $p$ depends on a  threshold \(\tau>0\);
    \item Finally, kernel regressors are trained on both the original and sparse training sets. Performance is assessed on the test set of size \(n_t < n\) using standard  metrics. In particular, we compute the Root Mean Square Error (RMSE) or the Root Mean Square Relative Error (RMSRE)  between the predicted values \(\hat{F} = \{\hat{f}_i\}_{i=1}^{n_t}\) and the ground truth \({F} = \{{f}_i\}_{i=1}^{n_t}\):
\[
\text{RMSE}(F, \hat{F}) = \sqrt{\dfrac{1}{n_t} \sum_{i=1}^{n_t} (f_i - \hat{f}_i)^2}, \quad \text{RMSRE}(F,\hat{F}) = \sqrt{\dfrac{1}{n_t}\sum_{i=1}^{n_t}    \biggl(\dfrac{f_i-\hat{f}_i}{f_i} \biggl)^2}.
\]

\end{enumerate}

Results can be reproduced with the free Python implementation available at \url{https://github.com/fabianacamattari/2L-FUSE}

\subsection{Experiments with synthetic data}

We consider synthetic datasets generated by randomly sampling \(n\) points from the \(d\)-dimensional unit cube \(\Omega = [0,1]^d\). As a first illustrative example, we set \(d = 35\) and \(n = 50000\), and define the target function as a Gaussian-like expression depending only on the first six components:
$$ f_1(\boldsymbol{x}) = {\rm e}^{-\sum_{j=1}^6(x_j-0.5)^2}. $$
Clearly, this function represents a controlled toy model specifically designed to validate the feature reduction capabilities of the 2L-FUSE method.

The second scenario involves functions with more heterogeneous dependencies among the input features, whose contributions and relevance vary depending on the value of the parameter \(\alpha \in \{-2,-1,0,1,2\}\). In this case we set \(d=15\), \(n=5000\) and define
$$
f_{2}^\alpha(\boldsymbol{x}) = {\rm e}^{x_1^2} + {\rm e}^{x_2} + 3x_3 + \cos(x_4 x_5) + 4x_6^2 + 10^\alpha \sum_{j=7}^{8} \log(x_j + 2) + 10^{-8} \sum_{j=9}^{15} x_j .
$$
Since the values of $f_2^\alpha$ may span several orders of magnitude depending on the value $\alpha$, the RMSRE offers a scale-invariant metric for comparing model performances across different target ranges. 
%Additionally, for the optimization phase, the \(f_2^\alpha\) values were rescaled to the range [0,1].

In both cases, we test three radial kernels of different orders of smoothness, precisely the $C^{\infty}$ Gaussian (GA), the $C^2$ and $C^0$ Matèrn functions, denoted by M2 and M0, respectively.

For the first test function, the eigenvalues of the optimized matrix $\mTheta$ are plotted in Figures \ref{fig:1} and \ref{fig:2}, for the diagonal and non-diagonal case, respectively. In all plots, we set the eigenvalues below the machine precision equal to the latter. From those figures, we observe that, because of the typical ill-conditioning that characterizes the GA kernel, the decay of the associated eigenvalues is smoother than those of other kernels. Nevertheless, in the diagonal case with a threshold of  \(10^{-2}\), we can see that all the kernels identify the only features on which $f_1$ depends on (the first six variables).

As far as the non-diagonal case is concerned, from Figure  \ref{fig:2}, we note that only one linear combination of the original features, i.e., $\tilde{\boldsymbol{x}}_i = \left( \sum_{j=1}^d \sqrt{\lambda_1}v_1^j{{x}}_{i}^j\right)^{\intercal}$, $i=1,\ldots,n$, already explains the function $f_1$, meaning that the first (largest) eigenvalue is about four order of magnitude larger than all others. This claim is supported by Table \ref{tab:1}, where we compute the RMSE obtained by using all the features, the six variables identified by the optimization of the diagonal matrix $\mTheta$ and some linear combination of eigenvectors and original features properly scaled accordingly to the associated eigenvalues. Such Table confirms that, consistently with Figure \ref{fig:2}, one linear combination is sufficient to reconstruct the function $f_1$, indeed, with more features (the subsequent combinations) the RMSE saturates (see Table \ref{tab:1}).

\begin{figure}[htbp]
\centering
\begin{adjustbox}{width=\textwidth}
\begin{tikzpicture} 
\begin{groupplot}[
  group style={group size=1 by 4, vertical sep=-4.cm},
  width=\textwidth,
  height=0.4\textwidth,
  ymode=log,
  ymin=2.22e-16,
  ymax = 3.7791e+00,
  xmin=1,
  xmax=35,
  xtick=\empty,
  ylabel={Eigenvalues},
  grid=major,
legend style={at={(1,1)}, anchor=north east, column sep=1ex, font=\footnotesize}
]
% === Primo grafico (GA) ===
\nextgroupplot[]
\addplot+[blue, mark=*, mark options={fill opacity=0.5, opacity=0.5}] coordinates {
(1, 8.0293e-01) (2, 7.9005e-01) (3, 7.8131e-01) (4, 7.7204e-01) (5, 7.6616e-01) (6, 7.2093e-01)
(7, 1.8452e-03) (8, 1.3568e-03) (9, 1.3279e-03) (10, 6.2991e-04) (11, 4.0493e-04) (12, 3.0055e-04)
(13, 2.3557e-04) (14, 1.4867e-04) (15, 1.4316e-04) (16, 1.2157e-04) (17, 6.5874e-05) (18, 2.8228e-05)
(19, 2.7183e-06) (20, 1.6129e-06) (21, 7.0089e-07) (22, 4.9622e-08) (23, 7.2795e-11) (24, 1.7071e-11)
(25, 2.9794e-13) (26, 2.2789e-14) (27, 2.2200e-16) (28, 2.2200e-16) (29, 2.2200e-16) (30, 2.2200e-16)
(31, 2.2200e-16) (32, 2.2200e-16) (33, 2.2200e-16) (34, 2.2200e-16) (35, 2.2200e-16)
};
%\addlegendentry{GA}
%% RISULTATI (NON IN ORDINE)
% 7.72039890e-01 8.02928329e-01 7.20932007e-01 7.66157568e-01 7.90046334e-01 7.81310380e-01 
% 1.61291064e-06 1.38219839e-36 2.71827162e-06 2.24468445e-27 2.35565269e-04 1.70713148e-11 
% 1.35680160e-03 6.61181943e-21 2.27894502e-14 2.82278324e-05 4.67052907e-29 6.29911723e-04 
% 1.43155994e-04 6.68862178e-40 4.04929393e-04 0.00000000e+00 4.96216970e-08 2.84803403e-39 
% 2.97939516e-13 3.00551095e-04 7.00888961e-07 7.27945551e-11 1.21566874e-04 1.32786517e-03 
% 1.48669991e-04 1.35853507e-23 1.24829667e-16 1.84522639e-03 6.58740537e-05
%% RISULTATI (IN ORDINE E CORRISPONDENTE INDICE)
%(1, 8.02928329e-01) (2, 7.90046334e-01) (3, 7.81310380e-01) (4, 7.72039890e-01) (5, 7.66157568e-01) (6, 7.20932007e-01) (7, 1.84522639e-03) (8, 1.35680160e-03) (9, 1.32786517e-03) (10, 6.29911723e-04) (11, 4.04929393e-04) (12, 3.00551095e-04) (13, 2.35565269e-04) (14, 1.48669991e-04) (15, 1.43155994e-04) (16, 1.21566874e-04) (17, 6.58740537e-05) (18, 2.82278324e-05) (19, 2.71827162e-06) (20, 1.61291064e-06) (21, 7.00888961e-07) (22, 4.96216970e-08) (23, 7.27945551e-11) (24, 1.70713148e-11) (25, 2.27894502e-14) (26, 2.97939516e-13) 
%(27, 1.24829667e-16) -> (27, 2.2200e-16) 
%(28, 6.61181943e-21) -> (28, 2.2200e-16)
%(29, 1.35853507e-23) -> (29, 2.2200e-16)
%(30, 2.24468445e-27) -> (30, 2.2200e-16)
%(31, 4.67052907e-29) -> (31, 2.2200e-16)
%(32, 1.38219839e-36) -> (32, 2.2200e-16)
%(33, 2.84803403e-39) -> (33, 2.2200e-16)
%(34, 6.68862178e-40) -> (34, 2.2200e-16)
%(35, 0.00000000e+00) -> (35, 2.2200e-16)

% === Secondo grafico (M2) ===
\addplot+[red, mark=square*, mark options={fill opacity=0.5, opacity=0.5}] coordinates {
(1, 3.7791e+00) (2, 3.7507e+00) (3, 3.7322e+00) (4, 3.7110e+00) (5, 3.6857e+00) (6, 3.6030e+00)
(7, 2.2200e-16) (8, 2.2200e-16) (9, 2.2200e-16) (10, 2.2200e-16) (11, 2.2200e-16) (12, 2.2200e-16)
(13, 2.2200e-16) (14, 2.2200e-16) (15, 2.2200e-16) (16, 2.2200e-16) (17, 2.2200e-16) (18, 2.2200e-16)
(19, 2.2200e-16) (20, 2.2200e-16) (21, 2.2200e-16) (22, 2.2200e-16) (23, 2.2200e-16) (24, 2.2200e-16)
(25, 2.2200e-16) (26, 2.2200e-16) (27, 2.2200e-16) (28, 2.2200e-16) (29, 2.2200e-16) (30, 2.2200e-16)
(31, 2.2200e-16) (32, 2.2200e-16) (33, 2.2200e-16) (34, 2.2200e-16) (35, 2.2200e-16)
};
%\addlegendentry{M2}
%% RISULTATI (NON IN ORDINE)
% 3.6856954 3.779122  3.6030169 3.7109783  3.750684  3.7322032 poi tutti 0

%% RISULTATI IN ORDINE

% === Terzo grafico (M0) ===
\addplot+[green, mark=triangle*, mark options={fill opacity=0.5, opacity=0.9}] coordinates {
(1, 6.7644e-01) (2, 6.7569e-01) (3, 6.6623e-01) (4, 6.4083e-01) (5, 6.3559e-01) (6, 6.2761e-01)
(7, 2.2200e-16) (8, 2.2200e-16) (9, 2.2200e-16) (10, 2.2200e-16) (11, 2.2200e-16) (12, 2.2200e-16)
(13, 2.2200e-16) (14, 2.2200e-16) (15, 2.2200e-16) (16, 2.2200e-16) (17, 2.2200e-16) (18, 2.2200e-16)
(19, 2.2200e-16) (20, 2.2200e-16) (21, 2.2200e-16) (22, 2.2200e-16) (23, 2.2200e-16) (24, 2.2200e-16)
(25, 2.2200e-16) (26, 2.2200e-16) (27, 2.2200e-16) (28, 2.2200e-16) (29, 2.2200e-16) (30, 2.2200e-16)
(31, 2.2200e-16) (32, 2.2200e-16) (33, 2.2200e-16) (34, 2.2200e-16) (35, 2.2200e-16)
};
%\addlegendentry{M0}
%% RISULTATI (NON IN ORDINE)
% 6.3558728e-01 6.7644113e-01 6.4082986e-01 6.6622508e-01  6.7569363e-01 6.2760514e-01  
% 0.0000000e+00 0.0000000e+00 0.0000000e+00 0.0000000e+00 5.4732196e-39 0.0000000e+00
% 0.0000000e+00 0.0000000e+00 0.0000000e+00 0.0000000e+00 0.0000000e+00 0.0000000e+00 
% 0.0000000e+00 0.0000000e+00 0.0000000e+00 1.2410089e-29 0.0000000e+00 0.0000000e+00
% 0.0000000e+00 0.0000000e+00 0.0000000e+00 0.0000000e+00 0.0000000e+00 0.0000000e+00 
% 0.0000000e+00 0.0000000e+00 0.0000000e+00 2.4426943e-19 0.0000000e+00
%% quindi:
% 6.3558728e-01 6.7644113e-01 6.4082986e-01 6.6622508e-01  6.7569363e-01 6.2760514e-01  
% 2.2200e-16 2.2200e-16 2.2200e-16 2.2200e-16 2.2200e-16 2.2200e-16
% 2.2200e-16 2.2200e-16 2.2200e-16 2.2200e-16 2.2200e-16 2.2200e-16 
% 2.2200e-16 2.2200e-16 2.2200e-16 2.2200e-16 2.2200e-16 2.2200e-16
% 2.2200e-16 2.2200e-16 2.2200e-16 2.2200e-16 2.2200e-16 2.2200e-16 
% 2.2200e-16 2.2200e-16 2.2200e-16 2.2200e-16 2.2200e-16
%% ordine: 2, 5, 4, 3, 1, 6

\nextgroupplot[
            axis y line=none,
            axis x line*=bottom,
            grid=none,
            xtick={1,...,35},
            xticklabels={2,5,6,1,4,3,34,13,30,18,21,26,11,31,19,29,35,16,9,7,27,23,28,12,25,15,33,14,32,10,17,8,24,20,22},
            tick label style={font=\footnotesize, rotate=45, text=blue!80!black}, 
            xtick distance = 0.029411765
]

\nextgroupplot[
            axis y line=none,
            axis x line*=bottom,
            grid=none,
            xtick={1,...,35},
            xticklabels={2,5,6,4,1,3,7,8,9,10,11,12,13,14,15,16,17,18,19,20,21,22,23,24,25,26,27,28,29,30,31,32,33,34,35},
            tick label style={font=\footnotesize, rotate=45, text=red!80!black}, 
            xtick distance = 0.029411765
]
\nextgroupplot[
            axis y line=none,
            axis x line*=bottom,
            grid=none,
            xtick={1,...,35},
            xlabel={Feature Labels},
            xticklabels={2,5,4,3,1,6,34,22,11,7,8,9,10,12,13,14,15,16,17,18,19,20,21,23,24,25,26,27,28,29,30,31,32,33,35},
            tick label style={font=\footnotesize, rotate=45, text=green!80!black}, 
            xtick distance = 0.029411765
]
\end{groupplot}
\end{tikzpicture}
\end{adjustbox}
\caption{Eigenvalues of the diagonal matrix $\mTheta$ for the function $f_1$. The colors refer to the kernels GA (blue), M2 (red) and M0 (green).}
\label{fig:1}
\end{figure}
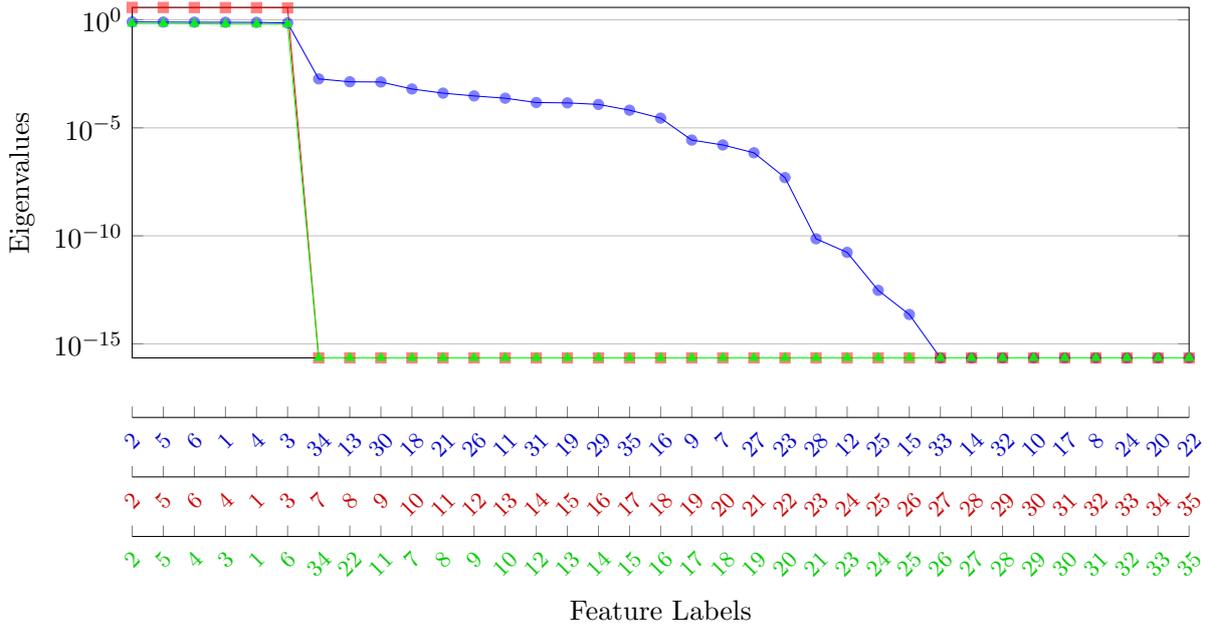

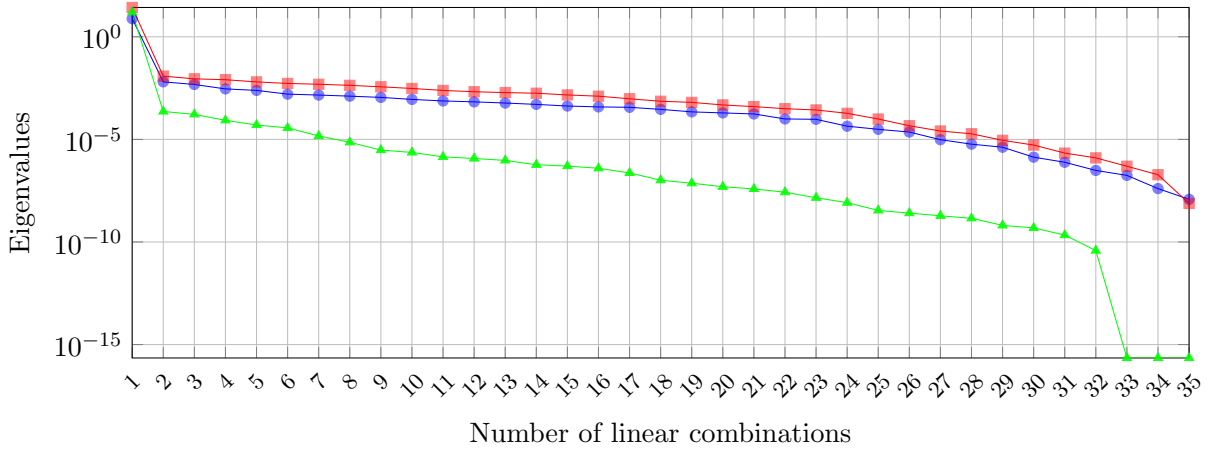
\begin{figure}[htbp]
\centering
\begin{adjustbox}{width=\textwidth}
\begin{tikzpicture} 
\begin{groupplot}[
  group style={group size=1 by 1, vertical sep=-2.4cm},
  width=\textwidth,
  height=0.4\textwidth,
  ymode=log,
  ymin=2.22e-16,
  ymax = 2.6898e+01,
  xmin=1,
  xmax=35,
  xtick={1,...,35},
  ylabel={Eigenvalues},
  xlabel = {Number of linear combinations},
  xticklabel style={font=\footnotesize,rotate=45},
  grid=major,
legend style={at={(1,1)}, anchor=north east, column sep=1ex, font=\footnotesize}
]
% === Primo grafico (GA) ===
\nextgroupplot[]
\addplot+[blue, mark=*, mark options={fill opacity=0.5, opacity=0.5}]  coordinates {
(1, 7.6596e+00) (2, 6.3804e-03) (3, 4.7472e-03) (4, 2.8871e-03) 
(5, 2.4424e-03) (6, 1.6058e-03) (7, 1.4483e-03) (8, 1.2801e-03) 
(9, 1.1134e-03) (10, 8.8887e-04) (11, 7.4704e-04) (12, 6.7430e-04)
(13, 5.9183e-04) (14, 5.0779e-04) (15, 4.2258e-04) (16, 3.8265e-04) 
(17, 3.6616e-04) (18, 2.9044e-04) (19, 2.1940e-04) (20, 1.9491e-04) 
(21, 1.7368e-04) (22, 9.9119e-05) (23, 9.4916e-05) (24, 4.4003e-05)
(25, 3.0968e-05) (26, 2.2809e-05) (27, 9.6585e-06) (28, 5.8551e-06) 
(29, 4.1504e-06) (30, 1.3592e-06) (31, 7.6794e-07) (32, 3.0703e-07) 
(33, 1.7598e-07) (34, 3.9963e-08) (35, 1.1989e-08)
};
%\addlegendentry{GA}

% === Secondo grafico (M2) ===
\addplot+[red, mark=square*, mark options={fill opacity=0.5, opacity=0.5}]  coordinates {
(1, 2.6898e+01) (2, 1.2088e-02) (3, 8.9282e-03) (4, 8.0925e-03)
(5, 6.3336e-03) (6, 5.3317e-03) (7, 4.8612e-03) (8, 4.3142e-03) 
(9, 3.6529e-03) (10, 3.0319e-03) (11, 2.4367e-03) (12, 2.0973e-03)
(13, 1.9136e-03) (14, 1.7785e-03) (15, 1.4660e-03) (16, 1.2748e-03) 
(17, 9.5476e-04) (18, 7.2139e-04) (19, 6.3755e-04) (20, 4.7698e-04) 
(21, 3.9584e-04) (22, 3.1330e-04) (23, 2.7169e-04) (24, 1.8673e-04)
(25, 9.8223e-05) (26, 4.6212e-05) (27, 2.5558e-05) (28, 1.8736e-05) 
(29, 9.0123e-06) (30, 5.2948e-06) (31, 2.1487e-06) (32, 1.2527e-06) 
(33, 4.9160e-07) (34, 1.8877e-07) (35, 7.5832e-09)
};
%\addlegendentry{M2}

% === Terzo grafico (M0) ===

\addplot+[green, mark=triangle*, mark options={fill opacity=0.5, opacity=0.9}]  coordinates {
(1, 1.5711e+01) (2, 2.2911e-04) (3, 1.6549e-04) (4, 8.4770e-05) 
(5, 5.0174e-05) (6, 3.6338e-05) (7, 1.4717e-05) (8, 7.1849e-06) 
(9, 3.0280e-06) (10, 2.3096e-06) (11, 1.4064e-06) (12, 1.1713e-06)
(13, 9.4698e-07) (14, 5.7978e-07) (15, 5.0025e-07) (16, 3.9062e-07) 
(17, 2.3200e-07) (18, 1.0228e-07) (19, 7.3062e-08) (20, 4.9262e-08) 
(21, 3.8357e-08) (22, 2.6606e-08) (23, 1.4333e-08) (24, 8.2209e-09)
(25, 3.4588e-09) (26, 2.5439e-09) (27, 1.8914e-09) (28, 1.4223e-09) 
(29, 6.4247e-10) (30, 4.8569e-10) (31, 2.1623e-10) (32, 3.8132e-11) 
(33, 2.2200e-16) (34, 2.2200e-16) (35, 2.2200e-16)
};
%\addlegendentry{M0}
\end{groupplot}
\end{tikzpicture}
\end{adjustbox}
\caption{Eigenvalues of the full matrix $\mTheta$ for the function $f_1$. The colors refer to the kernels GA (blue), M2 (red) and M0 (green).} 
\label{fig:2}
\end{figure}

\begin{table}[h!]
\centering
\small
\renewcommand{\arraystretch}{1.2}
\begin{tabular}{c|c|c|c|}
%\multicolumn{1}{c}{} & \multicolumn{3}{c}{\textbf{}} \\
%\cline{2-4}
\multicolumn{1}{c}{} & \multicolumn{1}{c}{{GA}} & \multicolumn{1}{c}{{M2}} & \multicolumn{1}{c}{{M0}}\\
\cline{2-4}%\hline 
\({\boldsymbol{x}_i}=({x}_i^1,\ldots, {x}_i^d)\) %\(\text{RMSE}_\text{ALL}\) 
& 2.82e-01  
& 2.29e-01
& 2.35e-01 \\
\cline{2-4}%\hline 
\({\boldsymbol{x}_i}=(\varepsilon_1{x}_i^1,\ldots, \varepsilon_6{x}_i^6)\) %\(\text{RMSE}_\text{SEL}\) 
& 1.28e-02 
& 9.90e-03 
& 2.68e-02 \\
\cline{2-4}%\hline 
\(\tilde{\boldsymbol{x}}_i =\sqrt{\lambda_1}\boldsymbol{v}_1\boldsymbol{x}_i^\intercal\)%\(\text{RMSE}_{\lambda_1 \;\;\;}\) 
& 2.28e-03
& 1.17e-03
& 1.45e-03 \\
\cline{2-4}%\hline 

\noalign{\vskip-6pt}
\multicolumn{1}{c}{}%\(\text{RMSE}_{\lambda_1, \, \lambda_2\;}\) 
& \multicolumn{1}{c}{} 
& \multicolumn{1}{c}{} 
& \multicolumn{1}{c}{} \\
\cline{2-4}
\(\tilde{\boldsymbol{x}}_i = \left( \sqrt{\lambda_1}\boldsymbol{v}_1\boldsymbol{x}_i^\intercal,\sqrt{\lambda_2}\boldsymbol{v}_2\boldsymbol{x}_i^\intercal\right)\)%\(\text{RMSE}_{\lambda_1, \, \lambda_2\;}\) 
& 1.79e-03 
& 1.18e-03  
& 1.48e-03 
\\
%\cline{2-4}
\(\tilde{\boldsymbol{x}}_i = \left( \sqrt{\lambda_1}\boldsymbol{v}_1\boldsymbol{x}_i^\intercal,\sqrt{\lambda_2}\boldsymbol{v}_2\boldsymbol{x}_i^\intercal,\sqrt{\lambda_3}\boldsymbol{v}_3\boldsymbol{x}_i^\intercal\right)\)%\(\text{RMSE}_{\lambda_1, \, \lambda_2\;}\) 
& 1.79e-03  
& 1.19e-03 
& 1.49e-03 
\\
%\cline{2-4}
\multicolumn{1}{c}{\textcolor{black!30}{\(\vdots\)}}%\(\text{RMSE}_{\lambda_1, \, \lambda_2\;}\) 
& \multicolumn{1}{|c}{\textcolor{black!30}{\(\vdots\)}} 
& \multicolumn{1}{|c}{\textcolor{black!30}{\(\vdots\)}} 
& \multicolumn{1}{|c|}{\textcolor{black!30}{\(\vdots\)}} \\
%\cline{2-4}
\(\tilde{\boldsymbol{x}}_i = \left( \sqrt{\lambda_1}\boldsymbol{v}_1\boldsymbol{x}_i^\intercal,\ldots,\sqrt{\lambda_{10}}\boldsymbol{v}_{10}\boldsymbol{x}_i^\intercal\right)\)%\(\text{RMSE}_{\lambda_1, \, \lambda_2\;}\) 
& 1.79e-03 
& 1.30e-03 
& 1.41e-03 
\\
\cline{2-4}
\end{tabular}
\caption{RMSEs for the function 
\(f_1\) using different kernels and feature configurations: all available features, those selected through the optimization of a diagonal matrix \(\mTheta\), i.e., the first six, and finally those obtained through feature combinations, in the most general case including from one up to ten new linear combinations.}
\label{tab:1}
\end{table}

Similar results can be recovered for the function $f_2^\alpha$. Precisely, we first note that the feature importance in this case depends on the values of the parameter $\alpha$. Our 2L-FUSE method correctly interprets this behavior, see Figure \ref{fig:3}, where we plot the eigenvalues of $\mTheta$ in the diagonal case. We can clearly see that, as the parameter $\alpha$ grows, the features $x_7$ and $x_8$ became more and more relevant, for all the considered kernels. Similarly to the previous case, the results on the eigenvalues for $f_2^\alpha$ when the full matrix $\mTheta$ is optimized are reported in Figure \ref{fig:4}. We observe that only a few eigenvalues are significantly larger than the others (in some cases differing up to five orders of magnitude). This suggests that a small number of linear combinations, from one to about five in most cases, might be sufficient to represent the target. As confirmations, with a larger number of combinations, the scores do not improve further, they saturate. This is indeed confirmed by the results on the approximation of \(f_2^\alpha\), that are plotted in Figure \ref{fig:5}, where we  compare the performance obtained by using the original dataset (i.e., all available features) with those recovered by the 2L-FUSE method, for both the diagonal and non-diagonal case. In all scenarios, the highest performance is achieved using the selected subset of features. In the non-diagonal case, we observe that the errors decrease and saturate after considering only a few dimensions, stressing the fact that the proposed scheme can be efficiently used as a feature reduction scheme.

\begin{figure}[htbp]
\centering
\begin{adjustbox}{width=\textwidth}
\begin{tikzpicture} 
\begin{groupplot}[
  group style={group size=1 by 5, vertical sep=1.5cm 
  %1.4cm
  },
  width=\textwidth,
  height=0.25\textwidth,
  ymode=log,
  ymin=2.22e-16,
  ymax = 3.7791e+00,
  xmin=1,
  xmax=15,
  xtick = {1,...,15},
  xticklabel style={font=\footnotesize,rotate=45},
  yticklabel style={font=\footnotesize},
  grid=major,
  legend style={at={(1,1)}, 
  anchor=north east, 
  column sep=1ex, 
  font=\footnotesize}
]

% === Primo grafico: alpha -2 ===
\nextgroupplot[xticklabels={
    {\shortstack{\textcolor{blue}{6} \\ \textcolor{red}{6} \\ \textcolor{green}{6}}},
    {\shortstack{\textcolor{blue}{1} \\ \textcolor{red}{1} \\ \textcolor{green}{3}}},
    {\shortstack{\textcolor{blue}{2} \\ \textcolor{red}{2} \\ \textcolor{green}{1}}},
    {\shortstack{\textcolor{blue}{3} \\ \textcolor{red}{3} \\ \textcolor{green}{2}}},
    {\shortstack{\textcolor{blue}{5} \\ \textcolor{red}{5} \\ \textcolor{green}{10}}},
    {\shortstack{\textcolor{blue}{4} \\ \textcolor{red}{4} \\ \textcolor{green}{4}}},
    {\shortstack{\textcolor{blue}{10} \\ \textcolor{red}{11} \\ \textcolor{green}{5}}},
    {\shortstack{\textcolor{blue}{7} \\ \textcolor{red}{7} \\ \textcolor{green}{7}}},
    {\shortstack{\textcolor{blue}{8} \\ \textcolor{red}{9} \\ \textcolor{green}{15}}},
    {\shortstack{\textcolor{blue}{9} \\ \textcolor{red}{8} \\ \textcolor{green}{9}}},
    {\shortstack{\textcolor{blue}{11} \\ \textcolor{red}{10} \\ \textcolor{green}{13}}},
    {\shortstack{\textcolor{blue}{12} \\ \textcolor{red}{12} \\ \textcolor{green}{8}}},
    {\shortstack{\textcolor{blue}{13} \\ \textcolor{red}{13} \\ \textcolor{green}{11}}},
    {\shortstack{\textcolor{blue}{14} \\ \textcolor{red}{14} \\ \textcolor{green}{12}}},
    {\shortstack{\textcolor{blue}{15} \\ \textcolor{red}{15} \\ \textcolor{green}{14}}}
},ymax = 7.1366e-01]
\addplot+[blue, mark=*, mark options={fill opacity=0.5, opacity=0.5}]  coordinates {
(1, 3.0137e-01)
(2, 2.2970e-01)
(3, 1.2694e-01)
(4, 2.0624e-02)
(5, 8.6760e-04)
(6, 7.7321e-04)
(7, 9.0020e-13)
(8, 2.2200e-16)
(9, 2.2200e-16)
(10, 2.2200e-16)
(11, 2.2200e-16)
(12, 2.2200e-16)
(13, 2.2200e-16)
(14, 2.2200e-16)
(15, 2.2200e-16)
};
%\addlegendentry{GA}

\addplot+[red, mark=square*, mark options={fill opacity=0.5, opacity=0.5}] coordinates {
(1, 7.1366e-01)
(2, 5.2669e-01)
(3, 2.2962e-01)
(4, 5.8337e-02)
(5, 3.1184e-03)
(6, 2.5020e-03)
(7, 2.7609e-11)
(8, 9.1030e-12)
(9, 5.4704e-14)
(10, 2.2200e-16)
(11, 2.2200e-16)
(12, 2.2200e-16)
(13, 2.2200e-16)
(14, 2.2200e-16)
(15, 2.2200e-16)
};
%\addlegendentry{M2} 

\addplot+[green, mark=triangle*, mark options={fill opacity=0.5, opacity=0.9}] coordinates {
(1, 4.4236e-03)
(2, 1.6793e-03)
(3, 1.1508e-03)
(4, 8.9422e-04)
(5, 1.5735e-11)
(6, 1.3306e-15)
(7, 2.2200e-16)
(8, 2.2200e-16)
(9, 2.2200e-16)
(10, 2.2200e-16)
(11, 2.2200e-16)
(12, 2.2200e-16)
(13, 2.2200e-16)
(14, 2.2200e-16)
(15, 2.2200e-16)
};
%\addlegendentry{M0}

% === Secondo grafico: alpha -1 ===
\nextgroupplot[xticklabels={
    {\shortstack{\textcolor{blue}{6} \\ \textcolor{red}{6} \\ \textcolor{green}{6}}},
    {\shortstack{\textcolor{blue}{1} \\ \textcolor{red}{1} \\ \textcolor{green}{3}}},
    {\shortstack{\textcolor{blue}{3} \\ \textcolor{red}{2} \\ \textcolor{green}{1}}},
    {\shortstack{\textcolor{blue}{2} \\ \textcolor{red}{3} \\ \textcolor{green}{2}}},
    {\shortstack{\textcolor{blue}{8} \\ \textcolor{red}{7} \\ \textcolor{green}{7}}},
    {\shortstack{\textcolor{blue}{7} \\ \textcolor{red}{8} \\ \textcolor{green}{8}}},
    {\shortstack{\textcolor{blue}{4} \\ \textcolor{red}{5} \\ \textcolor{green}{9}}},
    {\shortstack{\textcolor{blue}{5} \\ \textcolor{red}{4} \\ \textcolor{green}{14}}},
    {\shortstack{\textcolor{blue}{9} \\ \textcolor{red}{9} \\ \textcolor{green}{15}}},
    {\shortstack{\textcolor{blue}{13} \\ \textcolor{red}{10} \\ \textcolor{green}{5}}},
    {\shortstack{\textcolor{blue}{12} \\ \textcolor{red}{11} \\ \textcolor{green}{4}}},
    {\shortstack{\textcolor{blue}{10} \\ \textcolor{red}{12} \\ \textcolor{green}{10}}},
    {\shortstack{\textcolor{blue}{11} \\ \textcolor{red}{13} \\ \textcolor{green}{13}}},
    {\shortstack{\textcolor{blue}{14} \\ \textcolor{red}{14} \\ \textcolor{green}{11}}},
    {\shortstack{\textcolor{blue}{15} \\ \textcolor{red}{15} \\ \textcolor{green}{12}}}
}, 
ymax = 7.2e-01, 
]

\addplot+[blue, mark=*, mark options={fill opacity=0.5, opacity=0.5}] coordinates {
(1, 2.8268e-01)
(2, 2.1179e-01)
(3, 1.8924e-02)
(4, 9.6326e-03)
(5, 2.5399e-03)
(6, 2.1557e-03)
(7, 8.1345e-04)
(8, 7.4838e-04)
(9, 5.4591e-06)
(10, 4.4268e-09)
(11, 2.2200e-16)
(12, 2.2200e-16)
(13, 2.2200e-16)
(14, 2.2200e-16)
(15, 2.2200e-16)
};
%\addlegendentry{GA}

\addplot+[red, mark=square*, mark options={fill opacity=0.5, opacity=0.5}] coordinates {
(1, 7.1313e-01)
(2, 5.1204e-01)
(3, 2.2047e-01)
(4, 5.8063e-02)
(5, 8.3032e-03)
(6, 8.2491e-03)
(7, 3.1930e-03)
(8, 2.4010e-03)
(9, 2.2200e-16)
(10, 2.2200e-16)
(11, 2.2200e-16)
(12, 2.2200e-16)
(13, 2.2200e-16)
(14, 2.2200e-16)
(15, 2.2200e-16)
};
%\addlegendentry{M2}

\addplot+[green, mark=triangle*, mark options={fill opacity=0.5, opacity=0.9}] coordinates {
(1, 4.4072e-03)
(2, 1.7460e-03)
(3, 1.1472e-03)
(4, 9.3163e-04)
(5, 7.8759e-05)
(6, 6.0805e-05)
(7, 1.3990e-14)
(8, 2.3162e-16)
(9, 2.2200e-16)
(10, 2.2200e-16)
(11, 2.2200e-16)
(12, 2.2200e-16)
(13, 2.2200e-16)
(14, 2.2200e-16)
(15, 2.2200e-16)
};
%\addlegendentry{M0}

% === Terzo grafico: alpha 0 ===
\nextgroupplot[xticklabels={
    {\shortstack{\textcolor{blue}{6} \\ \textcolor{red}{6} \\ \textcolor{green}{6}}},
    {\shortstack{\textcolor{blue}{8} \\ \textcolor{red}{8} \\ \textcolor{green}{8}}},
    {\shortstack{\textcolor{blue}{7} \\ \textcolor{red}{7} \\ \textcolor{green}{7}}},
    {\shortstack{\textcolor{blue}{3} \\ \textcolor{red}{1} \\ \textcolor{green}{3}}},
    {\shortstack{\textcolor{blue}{2} \\ \textcolor{red}{3} \\ \textcolor{green}{2}}},
    {\shortstack{\textcolor{blue}{1} \\ \textcolor{red}{2} \\ \textcolor{green}{1}}},
    {\shortstack{\textcolor{blue}{5} \\ \textcolor{red}{5} \\ \textcolor{green}{15}}},
    {\shortstack{\textcolor{blue}{4} \\ \textcolor{red}{4} \\ \textcolor{green}{14}}},
    {\shortstack{\textcolor{blue}{9} \\ \textcolor{red}{9} \\ \textcolor{green}{5}}},
    {\shortstack{\textcolor{blue}{10} \\ \textcolor{red}{10} \\ \textcolor{green}{13}}},
    {\shortstack{\textcolor{blue}{11} \\ \textcolor{red}{11} \\ \textcolor{green}{10}}},
    {\shortstack{\textcolor{blue}{12} \\ \textcolor{red}{12} \\ \textcolor{green}{9}}},
    {\shortstack{\textcolor{blue}{13} \\ \textcolor{red}{13} \\ \textcolor{green}{4}}},
    {\shortstack{\textcolor{blue}{14} \\ \textcolor{red}{14} \\ \textcolor{green}{11}}},
    {\shortstack{\textcolor{blue}{15} \\ \textcolor{red}{15} \\ \textcolor{green}{12}}}
},
ymax = 5.9342e-01,
ylabel={Eigenvalues}
]

\addplot+[blue, mark=*, mark options={fill opacity=0.5, opacity=0.5}]  coordinates {
(1, 2.2650e-01)
(2, 1.4951e-01)
(3, 1.3756e-01)
(4, 1.0854e-02)
(5, 6.7745e-03)
(6, 5.7656e-03)
(7, 1.5548e-04)
(8, 1.4839e-04)
(9, 2.2200e-16)
(10, 2.2200e-16)
(11, 2.2200e-16)
(12, 2.2200e-16)
(13, 2.2200e-16)
(14, 2.2200e-16)
(15, 2.2200e-16)
};
%\addlegendentry{GA}

\addplot+[red, mark=square*, mark options={fill opacity=0.5, opacity=0.5}] coordinates {
(1, 5.9342e-01)
(2, 3.9750e-01)
(3, 3.9358e-01)
(4, 3.8255e-01)
(5, 4.5654e-02)
(6, 4.3240e-02)
(7, 1.2726e-03)
(8, 1.1537e-03)
(9, 7.5827e-09)
(10, 2.2200e-16)
(11, 2.2200e-16)
(12, 2.2200e-16)
(13, 2.2200e-16)
(14, 2.2200e-16)
(15, 2.2200e-16)
};
%\addlegendentry{M2}

\addplot+[green, mark=triangle*, mark options={fill opacity=0.5, opacity=0.9}] coordinates {
(1, 2.7209e-03)
(2, 1.6746e-03)
(3, 1.5031e-03)
(4, 1.2937e-03)
(5, 7.0952e-04)
(6, 6.2095e-04)
(7, 9.6558e-13)
(8, 2.2200e-16)
(9, 2.2200e-16)
(10, 2.2200e-16)
(11, 2.2200e-16)
(12, 2.2200e-16)
(13, 2.2200e-16)
(14, 2.2200e-16)
(15, 2.2200e-16)
};
%\addlegendentry{M0}

% === Quarto grafico: alpha 1 ===
\nextgroupplot[xticklabels={
    {\shortstack{\textcolor{blue}{7} \\ \textcolor{red}{7} \\ \textcolor{green}{8}}},
    {\shortstack{\textcolor{blue}{8} \\ \textcolor{red}{8} \\ \textcolor{green}{7}}},
    {\shortstack{\textcolor{blue}{6} \\ \textcolor{red}{6} \\ \textcolor{green}{6}}},
    {\shortstack{\textcolor{blue}{3} \\ \textcolor{red}{3} \\ \textcolor{green}{3}}},
    {\shortstack{\textcolor{blue}{2} \\ \textcolor{red}{2} \\ \textcolor{green}{2}}},
    {\shortstack{\textcolor{blue}{1} \\ \textcolor{red}{1} \\ \textcolor{green}{1}}},
    {\shortstack{\textcolor{blue}{15} \\ \textcolor{red}{15} \\ \textcolor{green}{12}}},
    {\shortstack{\textcolor{blue}{13} \\ \textcolor{red}{4} \\ \textcolor{green}{13}}},
    {\shortstack{\textcolor{blue}{4} \\ \textcolor{red}{5} \\ \textcolor{green}{9}}},
    {\shortstack{\textcolor{blue}{5} \\ \textcolor{red}{9} \\ \textcolor{green}{10}}},
    {\shortstack{\textcolor{blue}{9} \\ \textcolor{red}{10} \\ \textcolor{green}{4}}},
    {\shortstack{\textcolor{blue}{10} \\ \textcolor{red}{11} \\ \textcolor{green}{5}}},
    {\shortstack{\textcolor{blue}{11} \\ \textcolor{red}{12} \\ \textcolor{green}{11}}},
    {\shortstack{\textcolor{blue}{12} \\ \textcolor{red}{13} \\ \textcolor{green}{14}}},
    {\shortstack{\textcolor{blue}{14} \\ \textcolor{red}{14} \\ \textcolor{green}{15}}}
},
ymax = 1.1180e+00]

\addplot+[blue, mark=*, mark options={fill opacity=0.5, opacity=0.5}] coordinates {
    (1, 5.6860e-01)
    (2, 5.6530e-01)
    (3, 3.9906e-03)
    (4, 3.3142e-03)
    (5, 1.8539e-03)
    (6, 1.4762e-03)
    (7, 1.1396e-13)
    (8, 2.2200e-16)
    (9, 2.2200e-16)
    (10, 2.2200e-16)
    (11, 2.2200e-16)
    (12, 2.2200e-16)
    (13, 2.2200e-16)
    (14, 2.2200e-16)
    (15, 2.2200e-16)
};
%\addlegendentry{GA}

\addplot+[red, mark=square*, mark options={fill opacity=0.5, opacity=0.5}] coordinates {
    (1, 1.1180e+00)
    (2, 1.1129e+00)
    (3, 1.7784e-01)
    (4, 1.1593e-02)
    (5, 6.6889e-03)
    (6, 6.2503e-03)
    (7, 2.2200e-16)
    (8, 2.2200e-16)
    (9, 2.2200e-16)
    (10, 2.2200e-16)
    (11, 2.2200e-16)
    (12, 2.2200e-16)
    (13, 2.2200e-16)
    (14, 2.2200e-16)
    (15, 2.2200e-16)
};
%\addlegendentry{M2}

\addplot+[green, mark=triangle*, mark options={fill opacity=0.5, opacity=0.9}] coordinates {
    (1, 6.6108e-03)
    (2, 5.8323e-03)
    (3, 4.1117e-04)
    (4, 2.2279e-04)
    (5, 6.9026e-05)
    (6, 9.5562e-06)
    (7, 9.6758e-07)
    (8, 1.6252e-10)
    (9, 2.2200e-16)
    (10, 2.2200e-16)
    (11, 2.2200e-16)
    (12, 2.2200e-16)
    (13, 2.2200e-16)
    (14, 2.2200e-16)
    (15, 2.2200e-16)
};
%\addlegendentry{M0}

% === Quinto grafico: alpha 2 ===
\nextgroupplot[xticklabels=  {    
        {\shortstack{\textcolor{blue}{7} \\ \textcolor{red}{7} \\ \textcolor{green}{8}}},
        {\shortstack{\textcolor{blue}{8} \\ \textcolor{red}{8} \\ \textcolor{green}{7}}},
        {\shortstack{\textcolor{blue}{6} \\ \textcolor{red}{6} \\ \textcolor{green}{3}}},
        {\shortstack{\textcolor{blue}{3} \\ \textcolor{red}{3} \\ \textcolor{green}{5}}},
        {\shortstack{\textcolor{blue}{11} \\ \textcolor{red}{2} \\ \textcolor{green}{2}}},
        {\shortstack{\textcolor{blue}{13} \\ \textcolor{red}{1} \\ \textcolor{green}{11}}},
        {\shortstack{\textcolor{blue}{1} \\ \textcolor{red}{4} \\ \textcolor{green}{10}}},
        {\shortstack{\textcolor{blue}{2} \\ \textcolor{red}{5} \\ \textcolor{green}{6}}},
        {\shortstack{\textcolor{blue}{4} \\ \textcolor{red}{9} \\ \textcolor{green}{4}}},
        {\shortstack{\textcolor{blue}{5} \\ \textcolor{red}{10} \\ \textcolor{green}{1}}},
        {\shortstack{\textcolor{blue}{9} \\ \textcolor{red}{11} \\ \textcolor{green}{9}}},
        {\shortstack{\textcolor{blue}{10} \\ \textcolor{red}{12} \\ \textcolor{green}{14}}},
        {\shortstack{\textcolor{blue}{12} \\ \textcolor{red}{13} \\ \textcolor{green}{12}}},
        {\shortstack{\textcolor{blue}{14} \\ \textcolor{red}{14} \\ \textcolor{green}{13}}},
        {\shortstack{\textcolor{blue}{15} \\ \textcolor{red}{15} \\ \textcolor{green}{15}}}
        },
        ymax = 1.5084e+00, 
        xlabel = Feature Labels]

\addplot+[blue, mark=*, mark options={fill opacity=0.5, opacity=0.5}] coordinates {
    (1, 7.7397e-01)
    (2, 7.6287e-01)
    (3, 3.7053e-04)
    (4, 1.0294e-04)
    (5, 2.2200e-16)
    (6, 2.2200e-16)
    (7, 2.2200e-16)
    (8, 2.2200e-16)
    (9, 2.2200e-16)
    (10, 2.2200e-16)
    (11, 2.2200e-16)
    (12, 2.2200e-16)
    (13, 2.2200e-16)
    (14, 2.2200e-16)
    (15, 2.2200e-16)
};
%\addlegendentry{GA}

\addplot+[red, mark=square*, mark options={fill opacity=0.5, opacity=0.5}] coordinates {
    (1, 1.5084e+00)
    (2, 1.4725e+00)
    (3, 1.4723e-03)
    (4, 7.9507e-04)
    (5, 3.0534e-15)
    (6, 2.2200e-16)
    (7, 2.2200e-16)
    (8, 2.2200e-16)
    (9, 2.2200e-16)
    (10, 2.2200e-16)
    (11, 2.2200e-16)
    (12, 2.2200e-16)
    (13, 2.2200e-16)
    (14, 2.2200e-16)
    (15, 2.2200e-16)
};
%\addlegendentry{M2}

\addplot+[green, mark=triangle*, mark options={fill opacity=0.5, opacity=0.9}] coordinates {
    (1, 6.9349e-03)
    (2, 5.9765e-03)
    (3, 3.1971e-08)
    (4, 2.2200e-16)
    (5, 2.2200e-16)
    (6, 2.2200e-16)
    (7, 2.2200e-16)
    (8, 2.2200e-16)
    (9, 2.2200e-16)
    (10, 2.2200e-16)
    (11, 2.2200e-16)
    (12, 2.2200e-16)
    (13, 2.2200e-16)
    (14, 2.2200e-16)
    (15, 2.2200e-16)};
%\addlegendentry{M0}

\end{groupplot}
\end{tikzpicture}
\end{adjustbox}
\caption{Eigenvalues of the diagonal matrix $\mTheta$ for the test function $f_2^\alpha$ with different kernels and values of \(\alpha\). In each graph the colors refer to the kernels GA (blue), M2 (red) and M0 (green). From top to bottom: $\alpha = -2$, $\alpha = -1$, $\alpha = 0$, $\alpha = 1$ and $\alpha = 2$.}
\label{fig:3}
\end{figure}
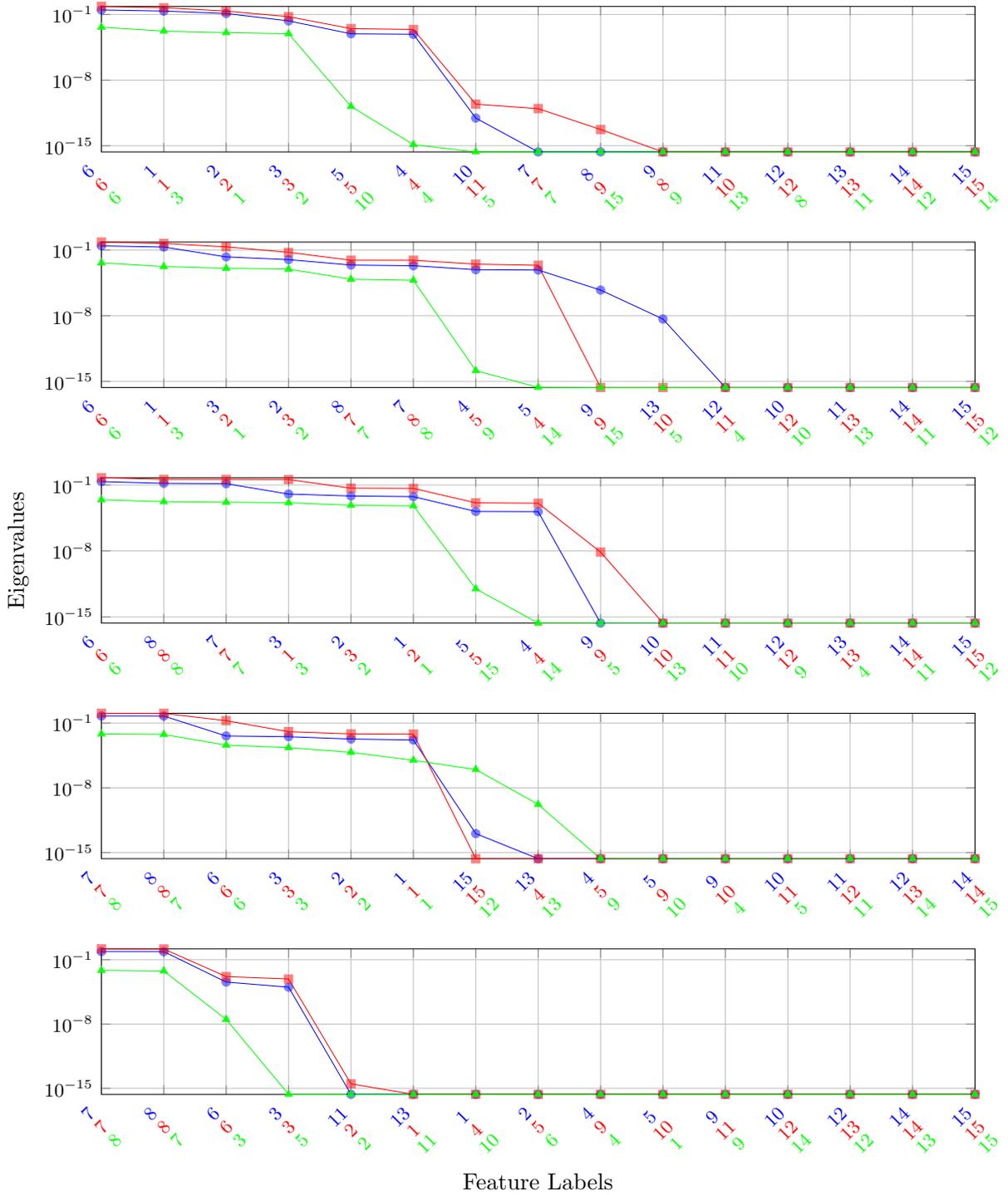

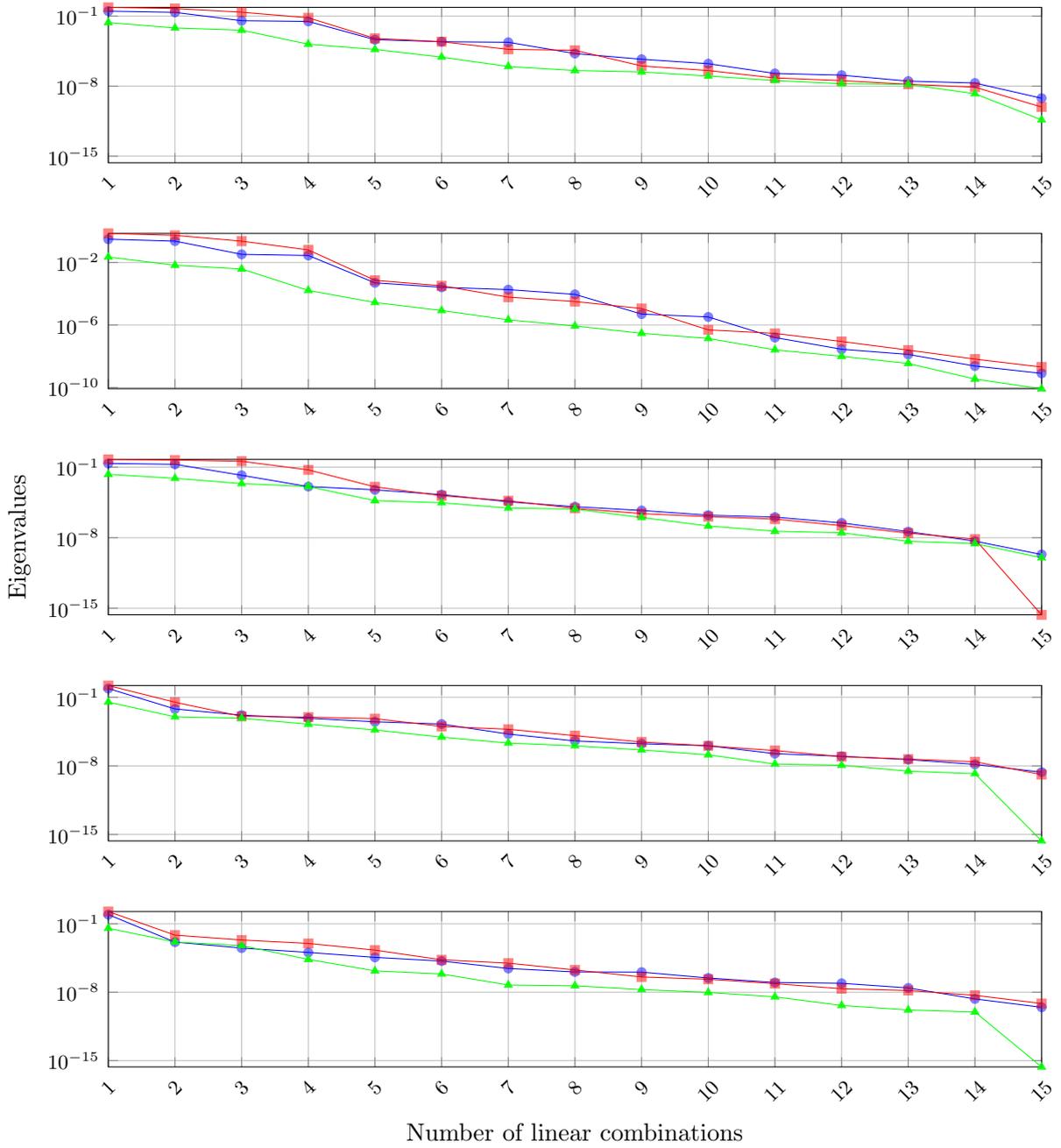
\begin{figure}[htbp]
\centering
\begin{adjustbox}{width=\textwidth}
\begin{tikzpicture} 
\begin{groupplot}[
  group style={group size=1 by 5, vertical sep=1.1cm
  %1.4cm
  },
  width=\textwidth,
  height=0.25\textwidth,
  ymode=log,
  ymin=2.22e-16,
  ymax = 3.7791e+00,
  xmin=1,
  xmax=15,
  xtick = {1,...,15},
  xticklabel style={font=\footnotesize,rotate=45},
  yticklabel style={font=\footnotesize},
  grid=major,
legend style={at={(1,1)}, anchor=north east, column sep=1ex, font=\footnotesize}
]

% === Primo grafico: alpha -2 ===
\nextgroupplot[ymax = 7.4192e-01]

\addplot+[blue, mark=*, mark options={fill opacity=0.5, opacity=0.5}]  coordinates {
(1, 3.0853e-01) (2, 2.3316e-01) (3, 3.4132e-02) 
(4, 2.9107e-02) (5, 4.3845e-04) (6, 2.7313e-04) 
(7, 2.4031e-04) (8, 1.8071e-05) (9, 4.8483e-06) 
(10, 1.7388e-06) (11, 1.8566e-07) (12, 1.2400e-07)
(13, 3.1738e-08) (14, 2.0399e-08) (15, 6.2163e-10)
};
%\addlegendentry{GA}

\addplot+[red, mark=square*, mark options={fill opacity=0.5, opacity=0.5}] coordinates {
(1, 7.4192e-01) (2, 5.7111e-01) (3, 2.4144e-01) 
(4, 6.7495e-02) (5, 5.8942e-04) (6, 2.7562e-04)  
(7, 4.7086e-05) (8, 3.8206e-05) (9, 1.0410846e-06) 
(10, 3.6152e-07) (11, 6.5667e-08) (12, 3.5605e-08) 
(13, 1.5449e-08) (14, 7.8828e-09) (15, 8.2300729e-11)};
%\addlegendentry{M2}

\addplot+[green, mark=triangle*, mark options={fill opacity=0.5, opacity=0.9}] coordinates {
(1, 2.2743e-02) (2, 6.5273e-03) (3, 3.8838e-03) 
(4, 1.5608e-04) (5, 4.6930e-05) (6, 7.9921e-06) 
(7, 9.2206e-07) (8, 3.6572e-07) (9, 2.6327e-07) 
(10, 1.0308e-07) (11, 3.6635e-08) (12, 1.7570e-08) 
(13, 1.5044e-08) (14, 1.7399e-09) (15, 4.2873e-12)
};
%\addlegendentry{M0}

% === Secondo grafico: alpha -1 ===
\nextgroupplot[ymax = 7.2442e-01, ymin = 8.3104e-11]

\addplot+[blue, mark=*, mark options={fill opacity=0.5, opacity=0.5}]  coordinates {
(1, 3.0710e-01) (2, 2.3166e-01) (3, 3.3040e-02) 
(4, 2.7998e-02) (5, 4.9609e-04) (6, 2.5787e-04) 
(7, 1.8293e-04) (8, 9.0754e-05) (9, 4.9831e-06) 
(10, 3.2686e-06) (11, 1.5799e-07) (12, 2.8233e-08) 
(13, 1.3102e-08) (14, 2.3807e-09) (15, 8.0687e-10)
};
%\addlegendentry{GA}

\addplot+[red, mark=square*, mark options={fill opacity=0.5, opacity=0.5}] coordinates {
(1, 7.2442e-01) (2, 5.4343e-01) (3, 2.3036e-01) 
(4, 6.3745e-02) (5, 7.2163e-04) (6, 3.2051e-04) 
(7, 6.0188e-05) (8, 3.1682e-05) (9, 1.1322e-05) 
(10, 4.8634e-07) (11, 2.9070e-07) (12, 8.8203e-08) 
(13, 2.4836e-08) (14, 6.5812e-09) (15, 2.0525e-09)
};
%\addlegendentry{M2}

\addplot+[green, mark=triangle*, mark options={fill opacity=0.5, opacity=0.9}] coordinates {
(1, 2.2231e-02) (2, 6.6117e-03) (3, 3.8218e-03)
(4, 1.6233e-04) (5, 2.7493e-05) (6, 8.4194e-06) 
(7, 2.1290e-06) (8, 8.6954e-07) (9, 2.9361e-07) 
(10, 1.3812e-07) (11, 2.5839e-08) (12, 9.9345e-09) 
(13, 3.4714e-09) (14, 3.5162e-10) (15, 8.3104e-11)
};
%\addlegendentry{M0}

% === Terzo grafico: alpha 0 ===
\nextgroupplot[ymax = 5.8144e-01, ylabel = Eigenvalues]

\addplot+[blue, mark=*, mark options={fill opacity=0.5, opacity=0.5}]  coordinates {
(1, 2.3328e-01) (2, 1.9140e-01) (3, 1.5610e-02) 
(4, 1.1795e-03) (5, 5.6543e-04) (6, 1.8789e-04) 
(7, 3.6388e-05) (8, 1.1887e-05) (9, 5.0019e-06) 
(10, 1.7003e-06) (11, 1.0991e-06) (12, 2.9310e-07) 
(13, 3.9670e-08) (14, 4.5243e-09) (15, 2.1437e-10)
};
%\addlegendentry{GA}

\addplot+[red, mark=square*, mark options={fill opacity=0.5, opacity=0.5}] coordinates {
(1, 5.8144e-01) (2, 5.1339e-01) (3, 3.8573e-01) 
(4, 5.3470e-02) (5, 1.1127e-03) (6, 1.4737e-04) 
(7, 4.5876e-05) (8, 8.2546e-06) (9, 2.4604e-06) 
(10, 1.2267e-06) (11, 7.0365e-07) (12, 1.5856e-07) 
(13, 2.8242e-08) (14, 7.0888e-09) (15, 2.2200e-16)
};
%\addlegendentry{M2}

\addplot+[green, mark=triangle*, mark options={fill opacity=0.5, opacity=0.9}] coordinates {
(1, 1.9400e-02) (2, 7.8613e-03) (3, 2.3470e-03) 
(4, 1.2026e-03) (5, 4.6929e-05) (6, 2.9403e-05) 
(7, 8.7387e-06) (8, 6.9712e-06) (9, 1.0032e-06) 
(10, 1.4517e-07) (11, 4.3634e-08) (12, 3.0676e-08) 
(13, 4.3392e-09) (14, 2.6045e-09) (15, 1.0252e-10)
};
%\addlegendentry{M0}

% === Quarto grafico: alpha 1 ===
\nextgroupplot[ymax = 1.5769e+00]

\addplot+[blue, mark=*, mark options={fill opacity=0.5, opacity=0.5}]  coordinates {
(1, 7.7679e-01) (2, 6.4102e-03) (3, 1.4986e-03) 
(4, 7.3088e-04) (5, 3.1733e-04) (6, 1.8542e-04) 
(7, 1.7591e-05) (8, 3.5113e-06) (9, 1.8263e-06) 
(10, 1.1349e-06) (11, 1.7702e-07) (12, 9.5142e-08) 
(13, 4.3929e-08) (14, 1.4295e-08) (15, 2.2954e-09)
};
%\addlegendentry{GA}

\addplot+[red, mark=square*, mark options={fill opacity=0.5, opacity=0.5}] coordinates {
(1, 1.5769e+00) (2, 3.1225e-02) (3, 1.1334e-03) 
(4, 9.4427e-04) (5, 6.5776e-04) (6, 1.0676e-04) 
(7, 5.4342e-05) (8, 1.2233e-05) (9, 2.7267e-06) 
(10, 1.1209e-06) (11, 3.6403e-07) (12, 8.7712e-08) 
(13, 4.8895e-08) (14, 2.6339e-08) (15, 1.3343e-09)
};
%\addlegendentry{M2}

\addplot+[green, mark=triangle*, mark options={fill opacity=0.5, opacity=0.9}] coordinates {
(1, 3.2360e-02) (2, 1.0167e-03) (3, 7.3586e-04) 
(4, 1.8036e-04) (5, 4.6665e-05) (6, 8.6308e-06) 
(7, 2.1413e-06) (8, 1.1245e-06) (9, 4.2681e-07) 
(10, 1.3523e-07) (11, 1.5457e-08) (12, 1.1640e-08)
(13, 2.8825e-09) (14, 1.6272e-09) (15, 2.2200e-16)
};
%\addlegendentry{M0}

% === Quinto grafico: alpha 2 ===
\nextgroupplot[ymax = 1.7954e+00, xlabel = Number of linear combinations]

\addplot+[blue, mark=*, mark options={fill opacity=0.5, opacity=0.5}]  coordinates {
(1, 8.7639e-01) (2, 1.3345e-03) (3, 3.2790e-04)
(4, 1.1755e-04) (5, 3.6456e-05) (6, 1.5282e-05)
(7, 2.6732e-06) (8, 1.1821e-06) (9, 1.1327e-06) 
(10, 2.8421e-07) (11, 9.8283e-08) (12, 8.1782e-08)
(13, 2.7365e-08) (14, 2.0519e-09) (15, 2.8510e-10)
};
%\addlegendentry{GA}

\addplot+[red, mark=square*, mark options={fill opacity=0.5, opacity=0.5}] coordinates {
(1, 1.7954e+00) (2, 6.9057e-03) (3, 2.1834e-03)
(4, 9.7714e-04) (5, 2.0172e-04) (6, 2.0764e-05) 
(7, 9.3846e-06) (8, 1.8847e-06) (9, 3.6001e-07) 
(10, 2.1061e-07) (11, 7.4992e-08) (12, 2.2645e-08)
(13, 1.4842e-08) (14, 4.8059e-09) (15, 6.9460e-10)
};
%\addlegendentry{M2}

\addplot+[green, mark=triangle*, mark options={fill opacity=0.5, opacity=0.9}] coordinates {
(1, 3.5175e-02) (2, 1.3641e-03) (3, 5.8521e-04)
(4, 2.2807e-05) (5, 1.5242e-06) (6, 7.3289e-07)
(7, 5.4609e-08) (8, 4.5159e-08) (9, 1.8656e-08)
(10, 9.4237e-09) (11, 3.3545e-09) (12, 4.4162e-10)
(13, 1.5284e-10) (14, 9.4831e-11) (15, 2.2200e-16)
};
%\addlegendentry{M0}

\end{groupplot}
\end{tikzpicture}
\end{adjustbox}
\caption{Eigenvalues of the full matrix $\mTheta$ for $f_2^\alpha$ with different kernels and values of \(\alpha\). In each graph the colors refer to the kernels GA (blue), M2 (red) and M0 (green). From top to bottom: $\alpha = -2$, $\alpha = -1$, $\alpha = 0$, $\alpha = 1$ and $\alpha = 2$.}
\label{fig:4}
\end{figure}

%%%%
%\begin{figure}[h!]
    %\centering
    %\includegraphics[width=1\linewidth]{New_Regression_Plots_ALL.png}
    %\caption{\textcolor{red}{BISOGNA RIFARLA E COMMENTARLA}\textcolor{magenta}{OK LA STO FACENDO SOTTO}{Comparison of regression results (RMSRE) for $f_2^\alpha$ across different kernels and values of \(\alpha\). From left to right: GA, M0 and M2. From top to bottom: $\alpha = -2$, $\alpha = -1$, $\alpha = 0$, $\alpha = 1$ and $\alpha = 2$. The red, green, and blue lines correspond to models using all available features, those selected via the diagonal version of the 2L algorithm, and those obtained by considering up to ten feature combinations, respectively.}}
    %\label{fig:5}
%\end{figure}
%%%%

%% RMSRE: f2 alpha
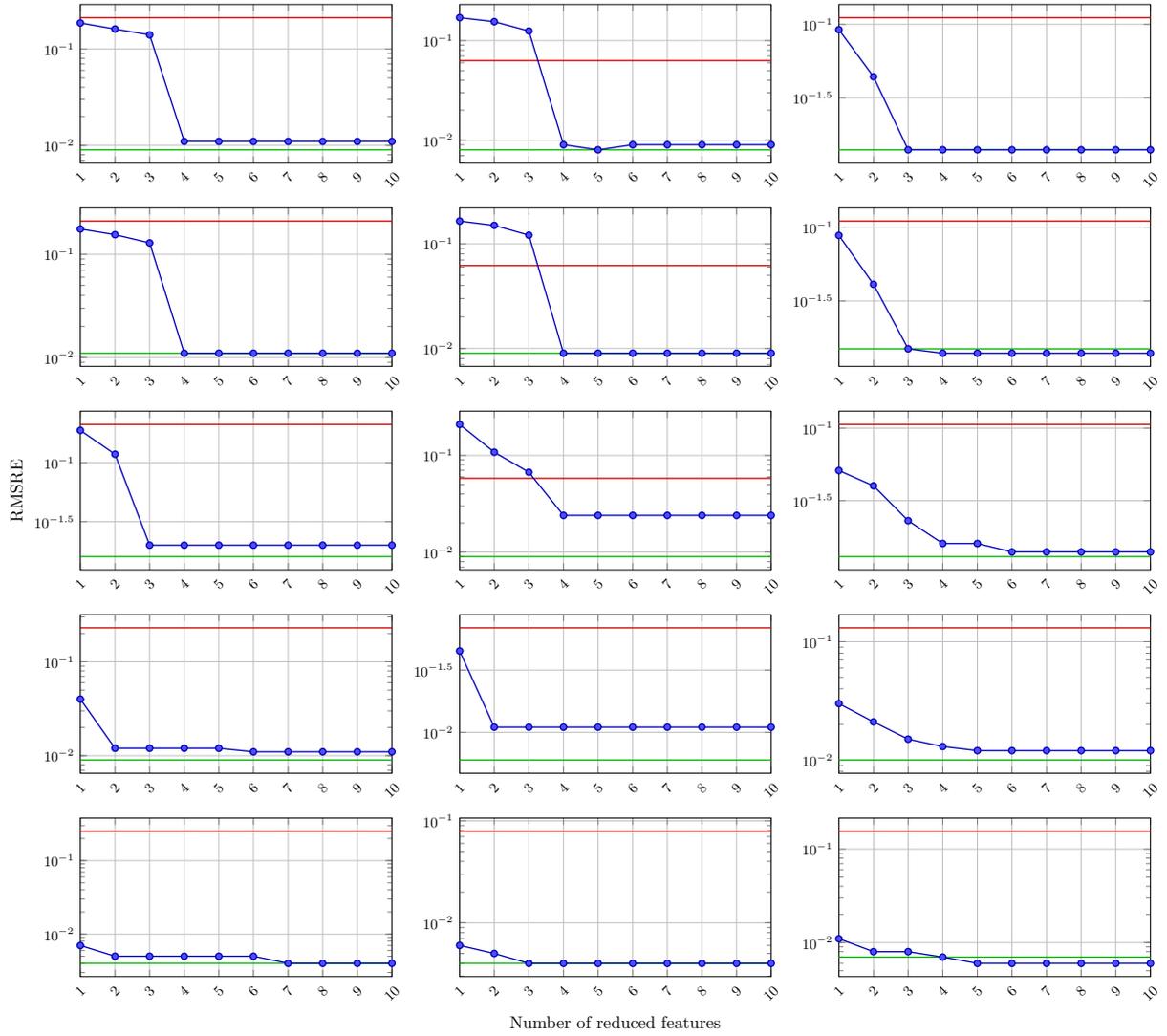
\begin{figure}[htbp]
\label{fig5}
\centering
\begin{adjustbox}{width=\textwidth}
\begin{tikzpicture} 
\begin{groupplot}[
  group style={group size=3 by 5, vertical sep=1.cm, horizontal sep=1.5cm},
  width=0.53\textwidth,
  height=0.32\textwidth,
  ymode=log,
  %ymin=2.22e-16,
  xmin=1,
  xmax=10,
  xtick={1,...,10},
  yticklabel style={font=\footnotesize},
  xticklabel style={font=\footnotesize,rotate=45},
  grid=major
]

% === GA, alpha -2 ===
\nextgroupplot[xticklabels={1,...,10}]
% ALL
\addplot+[red!80!black, mark=none, thick] coordinates {
(1, 0.211) (10, 0.211)};
% SEL
\addplot+[green!70!black, mark=none, thick] coordinates {
(1, 0.009) (10, 0.009)};
% COMBINATIONS
\addplot+[blue!70!black, mark=*, thick, mark options={fill=blue!70!white, draw=blue!70!black}] coordinates {
(1, 0.186)
(2, 0.161)
(3, 0.140)
(4, 0.011)
(5, 0.011)
(6, 0.011)
(7, 0.011)
(8, 0.011)
(9, 0.011)
(10, 0.011)
};

% === M2, alpha -2 ===
\nextgroupplot[xticklabels={1,...,10}]
% ALL
\addplot+[red!80!black, mark=none, thick] coordinates {
(1, 0.063) (10, 0.063)};
% SEL
\addplot+[green!70!black, mark=none, thick] coordinates {
(1, 0.008) (10, 0.008)};
% COMBINATIONS
\addplot+[blue!70!black, mark=*, thick, mark options={fill=blue!70!white, draw=blue!70!black}] coordinates {
(1, 0.170)
(2, 0.155)
(3, 0.125)
(4, 0.009)
(5, 0.008)
(6, 0.009)
(7, 0.009)
(8, 0.009)
(9, 0.009)
(10, 0.009)
};

% === M0, alpha -2 ===
\nextgroupplot[xticklabels={1,...,10}]
% ALL
\addplot+[red!80!black, mark=none, thick] coordinates {
(1, 0.111) (10, 0.111)};
% SEL
\addplot+[green!70!black, mark=none, thick] coordinates {
(1, 0.014) (10, 0.014)};
% COMBINATIONS
\addplot+[blue!70!black, mark=*, thick, mark options={fill=blue!70!white, draw=blue!70!black}] coordinates {
(1, 0.092)
(2, 0.044)
(3, 0.014)
(4, 0.014)
(5, 0.014)
(6, 0.014)
(7, 0.014)
(8, 0.014)
(9, 0.014)
(10, 0.014)
};

%%%%%%%%%%%%%%%%%%%%%%%%%%%%%%%%%%%%%

% === GA, alpha -1 ===
\nextgroupplot[xticklabels={1,...,10}]
% ALL
\addplot+[red!80!black, mark=none, thick] coordinates {
(1, 0.210) (10, 0.210)};
% SEL
\addplot+[green!70!black, mark=none, thick] coordinates {
(1, 0.011) (10, 0.011)};
% COMBINATIONS
\addplot+[blue!70!black, mark=*, thick, mark options={fill=blue!70!white, draw=blue!70!black}] coordinates {
(1, 0.176)
(2, 0.155)
(3, 0.129)
(4, 0.011)
(5, 0.011)
(6, 0.011)
(7, 0.011)
(8, 0.011)
(9, 0.011)
(10, 0.011)
};

% === M2, alpha -1 ===
\nextgroupplot[xticklabels={1,...,10}]
% ALL
\addplot+[red!80!black, mark=none, thick] coordinates {
(1, 0.062) (10, 0.062)};
% SEL
\addplot+[green!70!black, mark=none, thick] coordinates {
(1, 0.009) (10, 0.009)};
% COMBINATIONS
\addplot+[blue!70!black, mark=*, thick, mark options={fill=blue!70!white, draw=blue!70!black}] coordinates {
(1, 0.165)
(2, 0.150)
(3, 0.121)
(4, 0.009)
(5, 0.009)
(6, 0.009)
(7, 0.009)
(8, 0.009)
(9, 0.009)
(10, 0.009)
};

% === M0, alpha -1 ===
\nextgroupplot[xticklabels={1,...,10}]
% ALL
\addplot+[red!80!black, mark=none, thick] coordinates {
(1, 0.110) (10, 0.110)};
% SEL
\addplot+[green!70!black, mark=none, thick] coordinates {
(1, 0.015) (10, 0.015)};
% COMBINATIONS
\addplot+[blue!70!black, mark=*, thick, mark options={fill=blue!70!white, draw=blue!70!black}] coordinates {
(1, 0.088)
(2, 0.041)
(3, 0.015)
(4, 0.014)
(5, 0.014)
(6, 0.014)
(7, 0.014)
(8, 0.014)
(9, 0.014)
(10, 0.014)
};

%%%%%%%%%%%%%%%%%%%%%%%%%%%%%%%%%%%%%

% === GA, alpha 0 ===
\nextgroupplot[xticklabels={1,...,10}, ylabel = RMSRE]
% ALL
\addplot+[red!80!black, mark=none, thick] coordinates {
(1, 0.211) (10, 0.211)};
% SEL
\addplot+[green!70!black, mark=none, thick] coordinates {
(1, 0.016) (10, 0.016)};
% COMBINATIONS
\addplot+[blue!70!black, mark=*, thick, mark options={fill=blue!70!white, draw=blue!70!black}] coordinates {
(1, 0.188)
(2, 0.118)
(3, 0.020)
(4, 0.020)
(5, 0.020)
(6, 0.020)
(7, 0.020)
(8, 0.020)
(9, 0.020)
(10, 0.020)
};

% === M2, alpha 0 ===
\nextgroupplot[xticklabels={1,...,10}]
% ALL
\addplot+[red!80!black, mark=none, thick] coordinates {
(1, 0.058) (10, 0.058)};
% SEL
\addplot+[green!70!black, mark=none, thick] coordinates {
(1, 0.009) (10, 0.009)};
% COMBINATIONS
\addplot+[blue!70!black, mark=*, thick, mark options={fill=blue!70!white, draw=blue!70!black}] coordinates {
(1, 0.209)
(2, 0.108)
(3, 0.067)
(4, 0.024)
(5, 0.024)
(6, 0.024)
(7, 0.024)
(8, 0.024)
(9, 0.024)
(10, 0.024)
};

% === M0, alpha 0 ===
\nextgroupplot[xticklabels={1,...,10}]
% ALL
\addplot+[red!80!black, mark=none, thick] coordinates {
(1, 0.106) (10, 0.106)};
% SEL
\addplot+[green!70!black, mark=none, thick] coordinates {
(1, 0.013) (10, 0.013)};
% COMBINATIONS
\addplot+[blue!70!black, mark=*, thick, mark options={fill=blue!70!white, draw=blue!70!black}] coordinates {
(1, 0.051)
(2, 0.040)
(3, 0.023)
(4, 0.016)
(5, 0.016)
(6, 0.014)
(7, 0.014)
(8, 0.014)
(9, 0.014)
(10, 0.014)
};

%%%%%%%%%%%%%%%%%%%%%%%%%%%%%%%%%%%%%

% === GA, alpha 1 ===
\nextgroupplot[xticklabels={1,...,10}]
% ALL
\addplot+[red!80!black, mark=none, thick] coordinates {
(1, 0.230) (10, 0.230)};
% SEL
\addplot+[green!70!black, mark=none, thick] coordinates {
(1, 0.009) (10, 0.009)};
% COMBINATIONS
\addplot+[blue!70!black, mark=*, thick, mark options={fill=blue!70!white, draw=blue!70!black}] coordinates {
(1, 0.040)
(2, 0.012)
(3, 0.012)
(4, 0.012)
(5, 0.012)
(6, 0.011)
(7, 0.011)
(8, 0.011)
(9, 0.011)
(10, 0.011)
};

% === M2, alpha 1 ===
\nextgroupplot[xticklabels={1,...,10}]
% ALL
\addplot+[red!80!black, mark=none, thick] coordinates {
(1, 0.069) (10, 0.069)};
% SEL
\addplot+[green!70!black, mark=none, thick] coordinates {
(1, 0.006) (10, 0.006)};
% COMBINATIONS
\addplot+[blue!70!black, mark=*, thick, mark options={fill=blue!70!white, draw=blue!70!black}] coordinates {
(1, 0.045)
(2, 0.011)
(3, 0.011)
(4, 0.011)
(5, 0.011)
(6, 0.011)
(7, 0.011)
(8, 0.011)
(9, 0.011)
(10, 0.011)
};

% === M0, alpha 1 ===
\nextgroupplot[xticklabels={1,...,10}]
% ALL
\addplot+[red!80!black, mark=none, thick] coordinates {
(1, 0.131) (10, 0.131)};
% SEL
\addplot+[green!70!black, mark=none, thick] coordinates {
(1, 0.010) (10, 0.010)};
% COMBINATIONS
\addplot+[blue!70!black, mark=*, thick, mark options={fill=blue!70!white, draw=blue!70!black}] coordinates {
(1, 0.030)
(2, 0.021)
(3, 0.015)
(4, 0.013)
(5, 0.012)
(6, 0.012)
(7, 0.012)
(8, 0.012)
(9, 0.012)
(10, 0.012)
};

%%%%%%%%%%%%%%%%%%%%%%%%%%%%%%%%%%%%%

% === GA, alpha 2 ===
\nextgroupplot[xticklabels={1,...,10}]
% ALL
\addplot+[red!80!black, mark=none, thick] coordinates {
(1, 0.251) (10, 0.251)};
% SEL
\addplot+[green!70!black, mark=none, thick] coordinates {
(1, 0.004) (10, 0.004)};
% COMBINATIONS
\addplot+[blue!70!black, mark=*, thick, mark options={fill=blue!70!white, draw=blue!70!black}] coordinates {
(1, 0.007)
(2, 0.005)
(3, 0.005)
(4, 0.005)
(5, 0.005)
(6, 0.005)
(7, 0.004)
(8, 0.004)
(9, 0.004)
(10, 0.004)
};

% === M2, alpha 2 ===
\nextgroupplot[xticklabels={1,...,10}, xlabel=Number of reduced features]
% ALL
\addplot+[red!80!black, mark=none, thick] coordinates {
(1, 0.079) (10, 0.079)};
% SEL
\addplot+[green!70!black, mark=none, thick] coordinates {
(1, 0.004) (10, 0.004)};
% COMBINATIONS
\addplot+[blue!70!black, mark=*, thick, mark options={fill=blue!70!white, draw=blue!70!black}] coordinates {
(1, 0.006)
(2, 0.005)
(3, 0.004)
(4, 0.004)
(5, 0.004)
(6, 0.004)
(7, 0.004)
(8, 0.004)
(9, 0.004)
(10, 0.004)
};

% === M0, alpha 2 ===
\nextgroupplot[xticklabels={1,...,10}]
% ALL
\addplot+[red!80!black, mark=none, thick] coordinates {
(1, 0.155) (10, 0.155)};
% SEL
\addplot+[green!70!black, mark=none, thick] coordinates {
(1, 0.007) (10, 0.007)};
% COMBINATIONS
\addplot+[blue!70!black, mark=*, thick, mark options={fill=blue!70!white, draw=blue!70!black}] coordinates {
(1, 0.011)
(2, 0.008)
(3, 0.008)
(4, 0.007)
(5, 0.006)
(6, 0.006)
(7, 0.006)
(8, 0.006)
(9, 0.006)
(10, 0.006)
};

\end{groupplot}
\end{tikzpicture}
\end{adjustbox}
\caption{Comparison of the RMSEs for $f_2^\alpha$ with different kernels and values of \(\alpha\). From left to right: GA, M2 and M0. From top to bottom: $\alpha = -2$, $\alpha = -1$, $\alpha = 0$, $\alpha = 1$ and $\alpha = 2$. The red, green, and blue lines correspond to models using all available features, those selected via the diagonal version of the 2L-FUSE algorithm, and those obtained by considering up to ten feature combinations, respectively.}
\label{fig:5}
\end{figure}

\subsection{Experiments with real-world data}

In order to assess the robustness of the 2L-FUSE scheme, we consider an application within the field of solar physics, focusing on the prediction of a geomagnetic index known as SYM-H. This ground-based index provides a high-resolution measure of variations of Earth's magnetic field, and is therefore one of the key parameters used to understand the influence of solar activity on the Earth's magnetosphere. Geomagnetic storms are the main manifestations of solar activity at Earth and they are driven by the arrival of large masses of plasma and magnetic field structures originating from solar eruptions. In such conditions, the SYM-H index typically exhibits a rapid decrease, which constitutes the characteristic signature of a geomagnetic storm. The task consists in building a model capable to predict the SYM-H value one hour ahead from in-situ measurements of solar wind parameters collected by multiple spacecrafts.

The dataset at our disposal spans a period of approximately 15 years, from (January 1-st, 2005) to (December 31-st, 2019) resulting in \(N = 7888320\) 1-minute cadence points in a 14-dimensional space. Each sample consists of both measured and derived solar wind quantities, along with the corresponding value of the SYM-H index. The physical parameters are as follows: B,  B$_{\rm{x}}$, B$_{\rm{y}}$, and B$_{\rm{z}}$ [nT] (the magnetic field intensity and its three components); V$_{\rm{x}}$, V$_{\rm{y}}$, and V$_{\rm{z}}$ [Km/s] (the three coordinates of the solar wind velocity); $\rho$ [cm$^{-3}$] (the proton density number); T (the proton temperature); E$_{\rm{k}}$, E$_{\rm{m}}$, E$_{\rm{t}}$ (the kinetic, magnetic and total energies); and H$_{\rm{m}}$ (the magnetic helicity).

In this specific case, as an additional pre-processing step, beyond those already described, the data was resampled and aggregated by hours.

In this scenario we use the {Matérn} $C^0$ kernel, as it might be more suitable with real-world data affected by noise. In Table \ref{tab:rankingsolar} the results obtained by the use of 2L-FUSE in the diagonal case are presented: the most relevant features associated with non-zero eigenvalues  are listed in descending order. We can observe that the selected features are indeed the most representative from a physical viewpoint. Specifically, the vertical component B$_\mathrm{z}$ of the magnetic field vector is also significant, as it governs the magnetic reconnection process between the solar wind and Earth's magnetosphere, thus playing a crucial role in the transfer of energy from the solar wind to the Earth system. 
Moreover, this result is consistent with findings from previous research (see \cite{guastavinoetal2024,Camattari}). In those studies --- although conducted with slightly different purposes, within a classification setting and involving more complex deep learning architectures --- several feature selection algorithms identified B$_\mathrm{z}$ and V$_\mathrm{x}$ among the most relevant features in predicting the behavior of the SYM-H. Furthermore, the above mentioned papers show that discarding some redundancy, i.e., some parameters, generally leads better accuracy scores. The results obtained within the 2L-FUSE approach confirm that SYM-H is, as expected, predictive of its own value at the next hour, particularly in the main phase of the storm. Nevertheless, the physical features contain traces of transient solar structures that will generate the geomagnetic disturbance.
%\textbf{Qui si puo' anche dire che il SYMH è ovviamente predittivo di se stesso all'ora dopo, ma può essere anche soggetto a rapide variazioni, che non possono essere previste se non considerando i parametri caratterizzanti i transienti geoefficaci. Questo ha senso anche per "giustificare" il perchè escludiamo il SYMH dalle features quando ottimizziamo nel caso non diagonale.} 
%Moreover, this result is consistent with findings from previous research, although conducted in a classification setting and involving more complex deep learning architectures such as the Long Short-Term Memory network. In those studies it was demonstrated that the use of relevant features leads to an improvement in the performance score.
 %\textcolor{red}{(see citare paper Sabry e Daniele; Greedy?)}.

%Unlike the diagonal case, in the optimization process of \(\mSigma\) as a full matrix, the SYM-H variable is excluded from the input features, as the main interest is to explore how possible combinations of solar wind parameters alone can explain this quantity. %possible combinations of solar wind parameters that could most effectively model the target. 
%The SYM-H was subsequently reintroduced as an single feature alongside the linear combinations in the regression tests.} 

For the non-diagonal case, the RMSEs are reported in Figure \ref{fig:solar}. We observe that the use of the selected features outperforms the scores obtained using the full set of parameters, confirming that, even in real-world scenarios, some of them are indeed redundant. Moreover, unlike the simulated cases, the error increases as more linear combinations are considered, highlighting that a good approximation of the target can thus be achieved with a few combinations, whereas the subsequent ones introduce redundant information which may cause loss of accuracy. 
% \textbf{NON RICORDO SE AVEVAMO DETTO DI METTERE NEL TESTO GLI AUTOVALORI DEL CASO NON DIAGONALE, FORSE NO, NEL CASO SERVISSERO SONO QUI}\\
%Autovalori di M in ordine decrescente: 
%[6.70e+01, 2.52e+00, 2.11e+00, 1.59e+00, 1.10e+00, 1.07e+00, 9.29e-01, 8.24e-01, 5.14e-01, 4.08e-01, 2.75e-01, 5.98e-02, 1.51e-03]} 

\begin{table}[h!]
\centering
\small
\renewcommand{\arraystretch}{1.2}
\begin{tabular}{cc}
\multicolumn{1}{c}{\text{Feature}} & \multicolumn{1}{c}{\text{Eigenvalue}} \\ 
\hline
\hline
\multicolumn{1}{c}{SYM-H} & \multicolumn{1}{c}{
3.45e+02%%345.29 %345.29465
} \\

\multicolumn{1}{c}{B$_{\rm{z}}$} & \multicolumn{1}{c}{
1.78e+01%%17.75 %17.749613
} \\

\multicolumn{1}{c}{V$_{\rm{x}}$} & \multicolumn{1}{c}{
7.20e-01 %%0.72 %0.71951151
} \\

\multicolumn{1}{c}{B} & \multicolumn{1}{c}{
3.84e-01 %%0.38 %0.38407177
} \\

\multicolumn{1}{c}{T} & \multicolumn{1}{c}{
4.53e-02 %0.05 %0.04534645
} \\
\hline
\end{tabular}
\caption{Feature ranking based on the eigenvalues of diagonal matrix \(\mTheta\) for the solar dataset.}
\label{tab:rankingsolar}
\end{table}

%\begin{figure}
    %\centering
    %\includegraphics[width=0.5\linewidth]{Plot_Regression_Solar.png}
    %\caption{\color{blue} Comparison of regression results (RMSRE) for the solar dataset with all features (red line), selected ones (green) and up to all possible combinations (blue).
    %test set of the solar dataset computed on the original SYM-H values. \textbf{dire che 1 combinazione vuole dire: prima combo + SYM-H. ricorda che è stato tolto il SYM-H dall'ottimizzazione}}
    %\label{fig:enter-label}
%\end{figure}

%% PROVA FIGURA RMSE SOLARE - FABIANA
\begin{figure}[htbp]
\centering
\begin{adjustbox}{width=\textwidth}
\begin{tikzpicture} 
\begin{groupplot}[
  group style={group size=1 by 1, vertical sep=-2.4cm},
  width=\textwidth,
  height=0.4\textwidth,
  %ymode=log,
  %ymin=2.22e-16,
  %ymax = 2.6898e+01,
  xmin=1, %0,
  xmax=13, %14,
  xtick={1,...,13},
  ylabel={RMSE},
  xlabel = {Number of linear combinations},
  yticklabel style={font=\footnotesize},
  xticklabel style={font=\footnotesize,rotate=45},
  grid=major,
legend style={at={(1,1)}, anchor=north east, column sep=1ex, font=\footnotesize}
]

% === Primo grafico ALL FEATURES ===
\nextgroupplot[]
\addplot+[red!80!black, mark=none, thick]  
coordinates {(1, 4.83016415769272) (13, 4.83016415769272)
};
%\addlegendentry{ALL}

% === Secondo grafico SELECTED FEATURES ===
\addplot+[green!70!black, mark=none, thick]  coordinates {
(1, 3.398554050943968) (13, 3.398554050943968)
};
%\addlegendentry{SELECTED}

% === Terzo grafico COMBINATIONS ===
\addplot+[blue!70!black, mark=*, thick, mark options={fill=blue!70!white, draw=blue!70!black} %mark options={fill opacity=0.5, opacity=0.5}
]  coordinates {
(1, 3.694692786682)
(2, 3.885948334523)
(3, 4.042534463073)
(4, 4.294973651399)
(5, 4.413573094152)
(6, 4.596133943469)
(7, 4.765399245768)
(8, 5.041760438978) 
(9, 5.102631203794)
(10, 5.176085280275)
(11, 5.234612847248)
(12, 5.244959785389)
(13, 5.244475909818)
};
%\addlegendentry{COMBINATIONS}
% === Quarto grafico COMBINATIONS SELECTED ===
%\addplot+[orange!70!black, mark=*, thick, mark options={fill=orange!70!white, draw=orange!70!black} %mark options={fill opacity=0.5, opacity=0.5}]  coordinates {(1, 3.516362287517) (2, 3.587622071462) (3, 3.652367828168) (4, 3.711299493308)};
%\addlegendentry{COMBINATIONS SELECTED}
\end{groupplot}
\end{tikzpicture}
\end{adjustbox}
\caption{Comparison the RMSEs for the solar dataset with  all available features, those selected via
the diagonal version of the 2L-FUSE algorithm, and those obtained by considering up to ten feature
combinations, respectively plotted in red, green and blue.
%IN ARANCIONE (ma forse togliamo?) quello con combinazioni delle feature selezionate (4) + SYM-H.
%% CASO COMBINAZIONI CON SOLE FEATURES SELEZIONATE + COLONNA DI SYM-H
% RMSE with new_dim: 1 (A optimization: nondiagA_withoutSYMH) - original dataset: 3.516362287516986
% RMSE with new_dim: 2 (A optimization: nondiagA_withoutSYMH) - original dataset: 3.587622071462155
% RMSE with new_dim: 3 (A optimization: nondiagA_withoutSYMH) - original dataset: 3.6523678281676983
% RMSE with new_dim: 4 (A optimization: nondiagA_withoutSYMH) - original dataset: 3.7112994933079304
}
\label{fig:solar}
\end{figure}
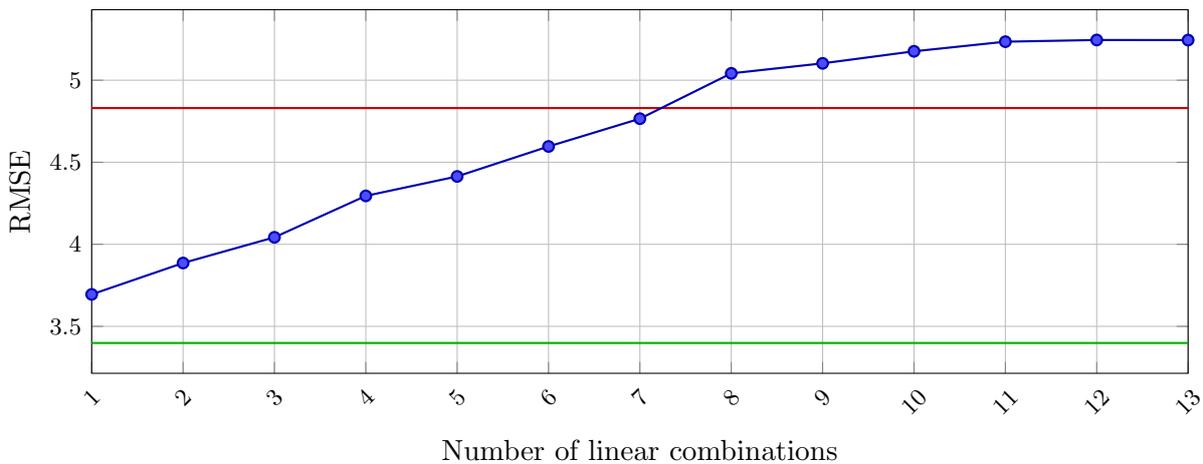

\section{Conclusions}
\label{conclusions}

In this work, we introduced the 2L-FUSE approach, a novel strategy for feature understanding and sparsity enhancement in kernel-based learning models. By learning a data-adaptive kernel metric through a two-layered kernel machine architecture, we were able to capture anisotropic and rotated structures in the input space, extending beyond the limitations of traditional isotropic kernel models.

From a theoretical viewpoint, we demonstrated that our approach is equivalent to performing interpolation in a transformed feature space, where the learned kernel metric governs the projection directions. We further showed that convergence guarantees, such as classical error bounds based on power functions and fill distances, still apply within this framework.

Empirical results on both synthetic and real-world datasets confirmed that 2L-FUSE can identify a minimal yet highly informative set of features without compromising predictive accuracy. Notably, in real-world scenarios like solar wind forecasting, the selected features aligned well with established physical interpretations, highlighting the interpretability of the method.

\section*{Acknowledgments}
Emma Perracchione, Fabiana Camattari and Sabrina Guastavino kindly acknowledge the support of the Fondazione Compagnia di San Paolo within the framework of the Artificial Intelligence Call for Proposals, AIxtreme project (ID Rol: 71708) and the GNCS-IN$\delta$AM PIANIS project ``Problemi Inversi e Approssimazione Numerica in Imaging Solare''. Emma Perracchione is further supported by the GOSSIP project (``Greedy Optimal Sampling for Solar Inverse Problems'') – funded by the Ministero dell’Universit\`a e della Ricerca (MUR) -  within the PRIN 2022 program, CUP: E53C24002330001. All the authors acknowledge Tizian Wenzel for the useful discussions.

\end{document}